\documentclass[lang = american]{ems-icm} 
\numberwithin{equation}{section}
\let\cal\mathcal
\def\Ascr{{\cal A}}
\def\Bscr{{\cal B}}

\def\Dscr{{\cal D}}
\def\Escr{{\cal E}}
\def\Fscr{{\cal F}}

\def\Hscr{{\cal H}}

\def\Lscr{{\cal L}}
\def\Mscr{{\cal M}}

\def\Oscr{{\cal O}}

\def\Rscr{{\cal R}}

\def\Tscr{{\cal T}}

\def\Wscr{{\cal W}}
\def\Xscr{{\cal X}}

\let\blb\mathbb
\let\Bbb\mathbb
\def\CC{{\blb C}}

\def\QQ{{\blb Q}}

\def \PP{{\blb P}}

\def \ZZ{{\blb Z}}

\def \NN{{\blb N}}
\def \RR{{\blb R}}

\let\oldmarginpar\marginpar
\def\marginpar#1{\oldmarginpar{\raggedright\tiny #1}}

\def\id{\text{id}}

\def\Res{\operatorname{Res}}

\def\Lotimes{\overset{L}{\otimes}}
\def\quot{/\mskip-1.8\thinmuskip/}

\def\mod{\operatorname{mod}}
\def\Gr{\operatorname{Gr}}

\def\Lie{\mathop{\text{Lie}}}

\def\Qch{\operatorname{Qch}}
\def\coh{\mathop{\text{\upshape{coh}}}}

\def\rad{\operatorname {rad}}

\def\Spec{\operatorname {Spec}}

\def\GL{\operatorname {GL}}

\def\depth{\operatorname {depth}}

\def\Ext{\operatorname {Ext}}
\def\Hom{\operatorname {Hom}}

\def\uEnd{\operatorname {\mathcal{E}\mathit{nd}}}
\def\End{\operatorname {End}}
\def\RHom{\operatorname {RHom}}
\def\uRHom{\operatorname {R\mathcal{H}\mathit{om}}}
\def\Sl{\operatorname {SL}}
\def\Sp{\operatorname {Sp}}

\def\End{\operatorname {End}}

\def\id{{\operatorname {id}}}

\def\add{\operatorname {add}}
\def\rk{\operatorname {rk}}

\def\gldim{\operatorname {gl\,dim}}

\def\r{\rightarrow}

\newtheorem{lemma}{Lemma}[section]
\newtheorem{proposition}[lemma]{Proposition}
\newtheorem{theorem}[lemma]{Theorem}
\newtheorem{corollary}[lemma]{Corollary}

\newtheorem{lemmas}{Lemma}[subsection]

\theoremstyle{definition}

\newtheorem{example}[lemma]{Example}
\newtheorem{definition}[lemma]{Definition}
\newtheorem{conjecture}[lemma]{Conjecture}
\newtheorem{question}[lemma]{Question}

{

}

\theoremstyle{remark}

\newtheorem{remark}[lemma]{Remark}

\newdimen\uboxsep \uboxsep=1ex
\def\uboxn#1{\vtop to 0pt{\hrule height 0pt depth 0pt\vskip\uboxsep
\hbox to 0pt{\hss #1\hss}\vss}}

\def\uboxs#1{\vbox to 0pt{\vss\hbox to 0pt{\hss #1\hss}
\vskip\uboxsep\hrule height 0pt depth 0pt}}

\def\ggd{\operatorname{gcd}}
\def\Perf{\operatorname{Perf}}

\def\Vol{\operatorname{Vol}}
\def\Sym{\operatorname{Sym}}
\def\HP{\operatorname{HP}}

\def\Refe{\operatorname{ref}}

\def\ggamma{{\fontencoding{OT2}\selectfont\text{\izhitsa}}}
\def\ggamma{\nu}
\def\Imm{\operatorname{Im}}

\begin{document}
\volume{?}
\title{Non-commutative crepant resolutions, an overview}
\titlemark{Non-commutative crepant resolutions}
\emsauthor{1}{Michel Van den Bergh}{M.~Van den Bergh}
\emsaffil{1}{Free University of Brussels\\ Pleinlaan 2\\ 1050 Brussel \email{michel.van.den.bergh@vub.be}}
\emsaffil{1}{Hasselt University\\ Martelarenlaan 42\\ 3500 Hasselt \email{michel.vandenbergh@uhasselt.be}}
\emsaffil{1}{Research Foundation - Flanders\\ Egmontstraat 5\\ 1000 Brussel}
\classification[14E30, 18E30]{14A22}
\begin{abstract}
  Non-commutative crepant resolutions (NCCRs) are non-commutative analogues of the usual crepant resolutions that appear in algebraic geometry.
In this paper we survey some results around NCCRs.
\end{abstract}
\maketitle
      \section{Introduction}
      In this paper we will give an introduction to non-commutative crepant resolutions with some emphasis on our joint work with \v{S}pela \v{S}penko
      about quotient singularities of reductive groups. Other surveys are \cite{Leuschke,Wemyss1,SpelaECM}.
\subsection{Notation and conventions}
We fix a few notations and definitions which are mostly self explanatory. For simplicity we assume throughout that  $k$ is an algebraically closed field of characteristic zero, although this is often not necessary.
In \S\ref{sec:stringy} we put $k=\CC$ when invoking Hodge theory. For us an \emph{algebraic variety} is a possibly singular integral separated scheme of finite type over $k$.

Modules over rings or sheaves of rings are left modules. Right modules are indicated by $(-)^\circ$.
If $\Lambda$ is a ring then we denote by $D(\Lambda)$ the unbounded derived category of complexes of $\Lambda$-modules and by $\Perf(\Lambda)$ its full subcategory
of perfect $\Lambda$-complexes. If $\Lambda$ is noetherian then we write $\mod(\Lambda)$ for the category of finitely generated
$\Lambda$-modules. We also put $\Dscr(\Lambda)=D^b(\mod(\Lambda))$.
We use similar notations in the geometric context.
If $\Xscr$ is an Artin stack and $\Lambda$ is a
quasi-coherent sheaf of rings on $\Xscr$ then $D_{\Qch}(\Lambda)$ is the unbounded derived category
of complexes of left $\Lambda$-modules with quasi-coherent cohomology. The category of perfect $\Lambda$-complexes is denoted by $\Perf(\Lambda)$
and we also put $\Dscr(\Lambda)=D^b(\coh(\Lambda))$ when $\Lambda$ is noetherian. If $\Lambda=\Oscr_{\Xscr}$ then we replace $\Lambda$ in the
notations by $\Xscr$.

A finitely generated~$R$-module $M$ over a normal noetherian domain is said to \emph{reflexive}
if the canonical map $M\mapsto M^{\vee\vee}$ is an isomorphism where $M^\vee=\Hom_R(M,R)$. This implies in particular that $M$ is torsion free.
If $R$ a commutative Noetherian domain then a \emph{maximal Cohen-Macaulay} $R$-module is an $R$-module $M$
such that $M_m$ is maximal Cohen-Macaulay as $R_m$-module for every maximal ideal $m$. If~$R$ is has
finite injective dimension then we say that~$R$ is \emph{Gorenstein}. This implies that $R$ is maximal Cohen-Macaulay.
\subsection{Crepant resolutions and derived equivalences}
\label{ssec:crepant}
 Let $X$ be a normal algebraic variety  with Gorenstein
singularities. A resolution of singularities $\pi:Y\r X$ is said to
\emph{crepant} if $\pi^\ast \omega_X=\omega_Y$. In some sense
a crepant resolution is the tightest possible smooth approximation of an algebraic variety. Such crepant resolutions need not exist however. 
For starters, their existence implies that
$X$ has rational singularities \cite[Corollary 5.24]{KM} and this already strong restriction is far from sufficient. For example the three-dimensional
hypersurface singularities
\begin{equation}
  \label{eq:An}
  x^2+y^2+z^2+w^n=0 \qquad (n\ge 2)
  \end{equation}
have crepant resolutions if and only if $n$ is even \cite[Corollary 1.16]{Reid}. 
When crepant resolution do exist they are generally not unique. E.g.
\begin{equation}
  \label{eq:atiyah}
  xy-zw=0,
\end{equation}
which corresponds to $n=2$ in \eqref{eq:An},
  has two distinct crepant resolutions given by blowing up $(x,z)$ and
  $(x,w)$. This is the so-called ``Atiyah flop''.

Nonetheless experience has shown that such different crepant resolutions
are strongly related. In particular we have the following result:
\begin{theorem}[\protect{\cite{Kon95,Batyrev1}}, see also \S\ref{sec:stringy} below] Assume that $X$ has canonical
  Gorenstein singularities. Then the Hodge numbers of $Y$ for a crepant resolution $Y\rightarrow X$ are independent
  of $Y$.
\end{theorem}
Kawamata and independently Bondal and Orlov in their lecture at ICM2002 conjectured an analogous categorical result of which we state a variant\footnote{We have omitted the projectivity hypotheses which appear in the original context and we have extended the conjecture to DM-stacks which is the natural context as will
  become clear below.}
  below.
\begin{conjecture}[\protect{\cite[Conjecture 4.4]{BonOrl}, \cite[Conjecture 1.2]{KawDK}}] \label{con:BO} Assume $X$ is a normal algebraic variety with Gorenstein singularities and $\pi_i:Y_i\rightarrow X$ for $i=1,2$ are two crepant resolutions (by schemes or DM-stacks).
  Then there is an equivalence of triangulated categories $F:\Dscr(Y_1)\cong \Dscr(Y_2)$, \emph{linear over $X$} (cfr Remark \ref{rem:linearity} below).
\end{conjecture}
The conjecture is known (under some probably unnecessary projectivity hypotheses) in dimension $\le 3$, by the work of Bridgeland \cite{Bridgeland}
(see \S\ref{sec:bridgeland} below), and for toric
varieties, by the work of Kawamata \cite{Kaw}. For symplectic singularities \cite{Beauville} it is true, up to an \'etale covering of $X$, by \cite[Theorem 1.6]{Kaledin}.
Furthermore it is known for many specific crepant resolutions, e.g.\ those related by variation of GIT
\cite{HLShipman,BFK,HLSam} (see also \S\ref{ssec:gitcrepant} below).

\begin{remark}
  \label{rem:notexplicit}
Conjecture \ref{con:BO} makes no statement about the nature of the equivalence $\Dscr(Y_1)\cong \Dscr(Y_2)$. In the case of the Atiyah flop
one possible equivalence is given by the Fourier-Mukai functor for the ``fiber product kernel'' $\Oscr_{Y_1\times_X Y_2}$ \cite[Theorem 3.6]{BondalOrlovSemi}
(see also \cite{bondalbodzenta}) but this is far from the
only possibility. Furthermore,
$\Oscr_{Y_1\times_X Y_2}$ does not always work as Example \ref{ex:cotangent} below shows.

It is now understood, thanks to intuition from mirror symmetry, that the equivalences in Conjecture \ref{con:BO} should be canonically associated
to  paths connecting two points in a topological space called the ``stringy K\"ahler moduli space'' (SKMS). 
In the case of the
Atiyah flop the SKMS is given by $\PP^1-\{0,1,\infty\}$ \cite{Donovan,HLSam}. See also \cite{HiranoWemyss} and \S\ref{sec:SKMS} below. The fact that
the asserted equivalence in Conjecture \ref{con:BO} is expected to be non-canonical by itself might be the reason that the conjecture seems difficult to prove. 
\end{remark}
Below $\Gr(d,n)$ is the Grassmannian of $d$-dimensional subspaces of the $n$-dimensional
vector space $k^n$.
\begin{example} \label{ex:cotangent} The cotangent bundles  $T^\ast \Gr(d,n)$ and $T^\ast \Gr(n-d,n)$, for  complementary Grassmannians with $d\le n/2$ are crepant resolutions
  of $\overline{B(d)}:=\{X\in M_n(k)\mid X^2=0, \rk X\le d\}$ (e.g.\ \cite[\S6.1]{CautisGrass}).
  According to \cite[\S6]{CautisGrass} there is an equivalence $F:\Dscr(T^\ast \Gr(d,n))\rightarrow \Dscr(T^\ast \Gr(n-d,n))$
but it is not given by the fiber product kernel (see \cite{Namikawa,Kawamata} for the case $(k,n)=(2,4)$).
\end{example}

\begin{remark} \label{rem:linearity}
As said one requires the derived equivalence $F$ in Conjecture \ref{con:BO} to be 
  linear over $X$. On the most basic level this means the following: let $\Perf(X)$ be the category of perfect complexes on $X$.
Then $\Dscr(Y_1)$, $\Dscr(Y_2)$ are
  $\Perf(X)$-modules, where $A\in \Perf(X)$ acts as $L\pi_{i}^\ast A\Lotimes_{Y_{i}}-$, for $i=1,2$, and we want 
  natural isomorphisms
  $F(L\pi^\ast_1A\Lotimes_{Y_1} -)\cong L\pi_2^\ast A \Lotimes_{Y_2} F(-)$ satisfying the appropriate compatibilities. \emph{To simplify the
    exposition we will implicitly assume in the rest of this paper that all constructions satisfy the appropriate linearity hypotheses.}
\end{remark}
\subsection{Non-commutative rings}
Most of the results below will be based on the interplay between algebraic geometry and non-commutative rings.
The relation between those subjects was first
observed by Beilinson \cite{Beilinson}. The connection is via tilting complexes.
\begin{definition} \label{def:tilting}
  Let $Y$ be a Noetherian scheme. A \emph{partial tilting complex} $\Tscr$ on~$Y$ is a perfect complex such that $\Ext^i_Y(\Tscr,\Tscr)=0$ for $i\neq 0$.
A \emph{tilting complex} is a partial tilting complex that \emph{generates} $D_{\Qch}(Y)$ in the sense that its right orthogonal is zero, i.e.\ $\RHom_Y(\Tscr,\Fscr)=0$ implies $\Fscr=0$. 
A \emph{(partial) tilting bundle} is a (partial) tilting complex which is a vector bundle.
\end{definition}
Below we will also use tilting complexes in slightly more general contexts (e.g.\ DM-stacks).
Very general results concerning tilting complexes are \cite[Theorems 1,2 ]{KellerKrause}.
For simplicity we state a slightly dumbed down version of them,
although below we will sometimes silently rely on the stronger results in loc.\ cit. See also \cite{Ri,Bondal2}. 
\begin{theorem}[\protect{\cite[Theorems 1,2 ]{KellerKrause}}] \label{th:kellerkrause} 
If $\Tscr$ is a tilting complex on a noetherian scheme~$Y$ then $\RHom_Y(\Tscr,-)$ defines 
an equivalence of categories between $D_{\Qch}(Y)$ and $D(\Lambda^\circ)$ for $\Lambda=\End_Y(\Tscr)$. Moreover, if $Y$ is regular
then $\Lambda$ has finite global dimension. If furthermore~$\Lambda$ is right noetherian
then $\RHom_Y(\Tscr,-)$ restricts to an equivalence of categories $\Dscr(Y)\cong \Dscr(\Lambda^\circ)$.
\end{theorem}
So a tilting complex reduces the homological algebra of $Y$ to the usually non-commutative ring $\Lambda=\End_Y(\Tscr)$.
In the case of projective space $\PP^n$ one can take $\Tscr=\Oscr\oplus\Oscr(1)\oplus\cdots \oplus\Oscr(n)$ \cite{Beilinson}.

\subsection{Bridgeland's result}
\label{sec:bridgeland}
\subsubsection{Flops}
\label{sec:flops}
Let us return to Conjecture \ref{con:BO}.
In the absence of any specific conjectural construction of the asserted derived equivalence (see Remark \ref{rem:notexplicit}) one may try to use the fact that
if $\pi_1$, $\pi_2$ are projective then $Y_1$, $Y_2$ are connected by a sequence of ``flops'' \cite[Theorem 1]{KawamataFlops},
so that it is then sufficient to prove the conjecture for flops. Recall that crepant resolutions $\pi_1:Y_1\rightarrow X$, $\pi_2:Y_2\rightarrow X$
form a flop if $X$ has terminal singularities \cite[Definition 2.12]{KM} and there is a line bundle $\Lscr$ on $Y_1$, relatively ample for $\pi_1$, such that the corresponding\footnote{This makes sense since $Y_1$ and $Y_2$ are isomorphic in codimension one.}  line bundle $\Lscr'$ on~$Y_2$ is anti-ample.

In \cite{Bridgeland}   Bridgeland proves that Conjecture \ref{con:BO} is true for three-dimensional flops (see also \cite{Chen}).
The key point is that the fibers of $\pi_1$, $\pi_2$ have  dimension $\le 1$. In the next section we explain a reinterpretation of Bridgeland's proof,
following \cite{VdB04a}.
\subsubsection{Maps with fibers of dimension $\le 1$}
\label{sec:1dim}
Assume that $\pi:Y\rightarrow X$ is a projective map between noetherian schemes.
We impose  the following conditions:
\begin{enumerate}
\item $R\pi_\ast \Oscr_Y=\Oscr_X$.
\item  The fibers  of $\pi$ have dimension $\le 1$.
\end{enumerate}
To simplify the discussion we will restrict ourselves furthermore to the case that $X=\Spec R$ is affine.\footnote{In \cite{VdB04a} $X$ is assumed to be quasi-projective.}
It turns out that in this case $\coh(Y)$ contains a tilting bundle which is of the form $\Tscr:=\Oscr_Y\oplus \Tscr_0$ where $\Tscr_0$ is obtained as an 
extension
 \begin{equation}
   \label{eq:tilting}
  0\rightarrow \Oscr_Y^r \rightarrow \Tscr_0\rightarrow \Lscr\rightarrow 0
\end{equation}
where $\Lscr$ is an ample line bundle on $\Oscr_Y$ generated by global sections and 
\eqref{eq:tilting} is associated to an arbitrary finite set of generators of $H^1(Y,\Lscr^{-1})$ as $R$-module (see \cite[(3.1)]{VdB04a}).
\begin{remark} \label{rem:backward}
  Note that by hypotheses (1), $\Oscr_X$, $\Lscr$ are partial tilting bundles on $Y$ such that $\Oscr_X\oplus \Lscr$ generates $D_{\Qch}(Y)$ \cite[Lemma 3.2.2]{VdB04a}.
  Moreover (2) and the fact that $\Lscr$ is generated by global sections implies $\Ext^{>0}_Y(\Oscr_Y,\Lscr)=0$. Likewise (2)  implies $\Ext^{>1}_Y(\Lscr,\Oscr_Y)=0$.
  The construction of the tilting bundle $\Tscr$ is based on the principle of ``killing the remaining backward $\Ext^1$''   in the sequence $(\Oscr_X,\Lscr)$ by a so-called ``semi-universal extension''.
  This principle extends to longer sequences. See e.g.\ \cite[Lemma 2.4]{Hara}, \cite[Lemma 3.1]{HillePerling1}. See also \S3.4 below for another application.
\end{remark}
So if we put $\Lambda=\End_Y(\Tscr)$, then we have $\End_Y(\Tscr^\vee)=\Lambda^\circ$ and from Theorem \ref{th:kellerkrause}
we obtain equivalences
\begin{equation}
  \label{eq:equivalences}
\RHom_Y(\Tscr,-):  \Dscr(Y)\cong \Dscr(\Lambda^\circ),\qquad \RHom_Y(\Tscr^\vee,-):\Dscr(Y)\cong \Dscr(\Lambda).
\end{equation}
To understand \eqref{eq:equivalences}  we can ask what $\Lambda$ looks like.
\begin{example}
Consider again the Atiyah flop \eqref{eq:atiyah}. In this case $R=k[x,y,z,w]/(xy-zw)$.
  This is a toric singularity and one can check that its class group is $\ZZ$ with generator $I=(x,z)$. The inverse of $I$ is the fractional ideal
  $I^{-1}=x^{-1}(x,w)$. The ring $\Lambda$
 turns out to be the same (up to isomorphism) for both crepant resolutions of $\Spec R$:
  \begin{equation}
\label{eq:Lambda}    
    \Lambda=\begin{pmatrix}
      R&I\\
      I^{-1} & R
\end{pmatrix}.
\end{equation}
Interestingly $\Lambda$ is built up from the three indecomposable graded maximal Cohen-Macaulay $R$-modules: $R$, $I$ and $I^{-1}$. In particular $\Lambda$ is itself Cohen-Macaulay
as $R$-module. This last fact turns out to be true more generally.
\end{example}
\begin{theorem} \label{th:fibers} Assume that $X=\Spec R$ 
  is a normal Gorenstein variety.
Assume 
that there exists a projective crepant resolution of singularities  $\pi:Y\rightarrow X$  such that  
the dimensions 
of the fibers of $\pi$ are $\le 1$. Let $\Tscr$ be the tilting bundle defined above\footnote{As we have stated in \S\ref{ssec:crepant}, the fact that $X$ has a crepant
  resolution implies that it has rational singularities by \cite[Corollary 5.24]{KM}. Thus in particular $R\pi_\ast\Oscr_Y=\Oscr_X$.}
  and put $T=\Gamma(Y,\Tscr)$. Then $\Lambda=\End_Y(\Tscr)=\End_R(T)$.
Furthermore $\Lambda$ and $T$ are maximal Cohen-Macaulay $R$-modules.
\end{theorem}
\begin{proof} 
  The fact that $\Lambda=\End_Y(\Tscr)$ is maximal Cohen-Macaulay follows from \cite[Lemma 3.2.9]{VdB04a}
  (see also \cite[Theorem 1.5]{IW2}, stated as Theorem \ref{th:iw} below). $T$ is maximal Cohen-Macaulay because it is a direct summand of $\Lambda$ as $R$-modules.
  Functoriality yields
  a map $i:\Lambda\rightarrow \End_R(T)$ which is an isomorphism in codimension one (since the singular locus of~$X$ has codimension $\ge 2$, as $X$ is normal).
  Since $\Lambda$ is maximal Cohen-Macaulay, it is reflexive and hence~$i$ must be an isomorphism.
\end{proof}
This result applies in particular if $X$ has dimension $2$ or if it is of dimension 3 with terminal singularities since then the condition on the dimension of the fibers 
is automatic.

Let us now assume that $X$ in Conjecture \ref{con:BO} is 3-dimensional and $\pi_1$, $\pi_2$ form a flop (see \S\ref{sec:flops}). We will still be assuming that $X=\Spec R$ is affine for simplicity. For $i=1,2$ we then have tilting bundles $\Tscr_{i}$
on $Y_{i}$ defined via \eqref{eq:tilting}, using $\Lscr$ on $Y_1$ and $(\Lscr')^{-1}$ on $Y_2$ (see \S\ref{sec:flops} for $\Lscr, \Lscr'$).
Let $(\Lambda_{i})_{i=1,2}$ be the corresponding endomorphism rings.
In this case Conjecture \ref{con:BO} follows from
\[
  \Dscr(Y_1)\overset{\eqref{eq:equivalences}}{\cong} \Dscr(\Lambda_{1}^\circ),\qquad
  \Dscr(Y_2)\overset{\eqref{eq:equivalences}}{\cong} \Dscr(\Lambda_{2}), \qquad 
  \Lambda_1^\circ\overset{\text{Morita}}{\cong} \Lambda_2.
\]
The asserted Morita equivalence is obtained in \cite[\S4.4]{VdB04a} using the local structure of
3-dimensional terminal singularities (see \cite[Example 2.3]{Kollar}). Nowadays  we may use
\cite[Corollary 8.8]{IR} (see also \cite[Theorem 1.5]{IyamaWemyssNCBO}) combined with \cite[Theorem 1.5]{IW2} (stated as Theorem \ref{th:iw} below)
to obtain that in any case $\Lambda_1$, $\Lambda_2^\circ$ are derived equivalent.

\medskip

At the end of the day we find that the two crepant resolutions $Y_1$, $Y_2$ of $X$ are derived equivalent to the same non-commutative ring (either $\Lambda^\circ_1$
or $\Lambda_2$). It turns out to be fruitful to think of this intermediate non-commutative ring as a \emph{third crepant resolution} of $X$, or of $R$, namely a
\emph{non-commutative 
  crepant resolution.} 
\section{Non-commutative (crepant) resolutions}
\subsection{Generalities}
\label{sec:generalities}
Below $R$ is a normal noetherian domain with quotient field $K$. We denote by $\Refe(R)$ the category of reflexive $R$-modules
and if $\Lambda$ is a reflexive $R$-algebra then $\Refe(\Lambda)$ is the category of $\Lambda$-modules which are
reflexive as $R$-modules.
A
\emph{reflexive Azumaya algebra} \cite{Leb99}~$\Lambda$ is a reflexive $R$-algebra
which is Azumaya in codimension one. A reflexive Azumaya algebra~$\Lambda$ is said to be trivial if it is of the form $\End_R(M)$ for $M$ a
reflexive $R$-module. In that case $\Refe(R)$ and $\Refe(\Lambda)$ are
equivalent. This is a particular case of ``reflexive Morita
equivalence'' which is defined in the obvious way.
\begin{definition} \label{def:NCR}
  A \emph{twisted non-commutative resolution} of $R$  is a reflexive Azumaya
algebra $\Lambda$ over $R$ such that $\gldim \Lambda<\infty$. If $\Lambda$ is trivial then 
$\Lambda$ is said to be a \emph{non-commutative resolution (NCR)} of $R$.
\end{definition}
\begin{definition} \label{ref-4.1-8}
  Assume that $R$ is Gorenstein.
A \emph{twisted non-commutative crepant resolution}~$\Lambda$ of $R$ 
is a twisted NCR of $R$ which is in addition a Cohen-Macaulay $R$-module.
If $\Lambda$ is an NCR then such a $\Lambda$ is said to be a \emph{non-commutative crepant resolution (NCCR)} of $R$.
\end{definition}
The point of these definitions is that they provide
 reasonable non-commutative substitutes for
 ``regularity'', ``birationality'' and ``crepancy''. This is explained
 in more detail in \cite[\S4]{VdB04a}.
\begin{remark} We will sometimes use the concepts introduced in
  Definitions \ref{def:NCR}, \ref{ref-4.1-8} for
  schemes, possibly non-affine. It is then
  understood that they reduce to the affine concepts, when restricting
  to open affine subschemes.
\end{remark}
\begin{remark} In the sequel we will be mostly concerned with NCCRs and thus the other definitions are mainly provided for context. Twisted NCCRs are natural generalizations of NCCRs, but the good properties
  of NCCRs (sometimes conjectural) are usually not shared by twisted NCCRs. See e.g.\ Example \ref{ex:twisted} below.
  The definition of a (twisted) NCR is more tentative.
  In particular the normality and reflexivity hypotheses do not seem very relevant. For example there is a nice theory of non-commutative resolutions of non-normal singularities in dimension one \cite{Leuschke}.
\end{remark}
\begin{example} \label{eq:NCCR}
  It follows from Theorem \ref{th:fibers}  and Theorem \ref{th:kellerkrause}
  that if  there exists a projective crepant resolution of singularities  $\pi:Y\rightarrow X$  such that  
the dimensions 
of the fibers of $\pi$ are $\le 1$ then $R$ has an NCCR.
\end{example}
We mention the following theorem which gives another indication that the definition of an NCCR is the ``correct one''.
\begin{theorem}[\protect{\cite[Theorem 1.5]{IW2}}] \label{th:iw}  Let $f : Y\rightarrow \operatorname{Spec} R$ be a projective birational morphism between
  Gorenstein varieties. Suppose that $Y$ is derived equivalent to some ring $\Lambda$,
then $f$ is a crepant resolution if and only if $\Lambda$ is an NCCR of $R$.
\end{theorem}
The following conjecture is a natural extension of Conjecture \ref{con:BO}.
\begin{conjecture}[\protect{\cite[Conjecture 4.6]{VdB32}}]
\label{ref-4.6-10}
 All crepant resolutions of $X$ (commutative as well as non-com\-mutative) are derived equivalent.
\end{conjecture}
We have the following result which is proved in the same way as the 3-dimensional McKay correspondence \cite{BKR}.
\begin{proposition}[\protect{\cite[Theorem 6.3.1, Proposition 6.2.1]{VdB32}}] \label{prop:iff}
  If $X$ has three-dimensional Gorenstein singularities and it has an NCCR $\Lambda$,
  then   it has a projective crepant resolution $Y\rightarrow X$ such that $\Lambda$ and $Y$ are derived equivalent.
  \end{proposition}
\begin{proposition}
Conjecture \ref{ref-4.6-10} is true if $X$ has dimension three, if we restrict to projective crepant resolutions.
\end{proposition}
\begin{proof} If $X$ has an NCCR $\Lambda$ then by Proposition \ref{prop:iff} $\Lambda$ is derived equivalent to a crepant resolution.  Hence we are reduced to Bridgeland's result (see \S\ref{sec:flops}). Alternatively, to have a very nice direct argument that any two NCCRs are   derived equivalent in dimension three
  we may use \cite[Corollary 8.8]{IR} (see also \cite[Theorem 1.5]{IyamaWemyssNCBO}).
  \end{proof}
Proposition \ref{prop:iff} is false for arbitrary three-dimensional Gorenstein singularities as was shown by Dao \cite{Dao}.
\begin{proposition}[\protect{\cite[Theorem 3.1, Remark 3.2]{Dao}}]  Assume $S$ is a regular local ring which is equicharacteristic or unramified, $0\neq f\in S$ and $R=S/(f)$ is normal. If $\dim R=3$ and $R$ is factorial then $R$ has no NCCR.
\end{proposition}
\begin{example} \label{ex:dao} It turns out that there are 3-dimensional factorial hypersurface singularities that admit a crepant resolution.
  A concrete example
  is given by $R=k[[x_0,x_1,x_2,x_3]]/(x_0^4+x_1^3+x_2^3+x_3^3)$ \cite[Theorem A,B]{Lin}. In particular a crepant resolution of such $R$ does not admit a tilting
  complex by Theorem \ref{th:iw}.
\end{example}
If $X$ is a normal Gorenstein algebraic variety with a crepant resolution then it has rational singularities \cite[Corollary 5.24]{KM}. A similar
result is true for NCCRs. 
\begin{theorem}[\protect{\cite[Theorem 1.1]{stafford2008}}] Let $R$ be a normal finitely generated Gorenstein $k$-algebra. If $R$ has a twisted NCCR then it has rational singularities. 
\end{theorem}
The actual result proved in loc.\ cit.\ applies in a more general context and this has been further exploited in 
\cite{IngallsYasuda2,IngallsYasuda} (see also \cite[Corollary 1.7]{MR4158464}).

\begin{remark}
  \label{rem:mma}
  In order to deal with singularities with a singular minimal model Iyama and Wemyss generalize the definition of an NCCR \cite{IW,IW2,Wemyss}
  to certain rings, of possibly infinite global dimension, called \emph{maximal modification algebras} (MMAs).
Remarkably, many of the results about NCCRs extend to MMAs. However in this overview we will restrict ourselves for simplicity to NCCRs.
\end{remark}

\subsection{Relation with crepant categorical resolutions}
\label{sec:ccr}
We conjecture that non-commutative crepant resolutions are examples of ``strongly crepant categorical resolutions'' as introduced by Kuznetsov in \cite{Kuz}.
However we can only prove this in special cases.

\medskip

Let $X$ be an algebraic variety. A \emph{categorical} resolution  \cite{Kuz} of $\Dscr(X)$ is a ``smooth'' triangulated category $\tilde{\Dscr}$  together with
functors
\[
  \pi_\ast:\tilde{\Dscr}\rightarrow \Dscr(X), \qquad \pi^\ast:\Perf(X)\rightarrow \tilde{\Dscr}
  \]
  which are adjoint (i.e. $\Hom_{\tilde{\Dscr}}(\pi^\ast A,B)\cong \Hom_{\Dscr(X)}(A,\pi_\ast B)$ for $A\in \Perf(X)$, $B\in \tilde{\Dscr}$) such that the natural transformation $\id_{\Perf(X)}\rightarrow \pi_\ast\pi^\ast$, obtained by putting $B=\pi^\ast A$, is an isomorphism. This implies in particular
  that $\pi^\ast$ is fully faithful.
  There is some variation possible in the definition of smoothness. For us it means that $\tilde{\Dscr}$ is equivalent to the derived category of perfect modules over a smooth DG-algebra \cite[Definition 2.23]{KKP}.
\begin{remark}
  If $\pi:Y\rightarrow X$ is a resolution of singularities of $X$ then $(\Dscr(Y),R\pi_\ast,L\pi^\ast)$  is a categorical resolution of $\Dscr(X)$ if and only
  if $X$ has rational singularities. Remarkably however, it has been shown in \cite{KuznetsovLunts} that $\Dscr(Y)$ can be suitably enlarged to yield a categorical resolution. On the other hand, this result cannot be extended to more general dg-categories \cite{MR4160877}.
 \end{remark}

  Following \cite{Kuz} we say that a categorical resolution $(\tilde{\Dscr},\pi_\ast,\pi^\ast)$ of $\Dscr(X)$
  is \emph{weakly crepant} if $\pi^\ast$ is both a left and a right adjoint to $\pi_\ast$.
  
  There is also a notion of a \emph{strongly crepant categorical} resolution for which we need the notion of a \emph{relative Serre functor}.
  To define this assume that $X$ is Gorenstein and that $\tilde{\Escr}$ is a smooth triangulated category which is a $\Perf(X)$-module.
  We will denote the action of $A\in \Perf(X)$ on $B\in \tilde{\Escr}$ as $A\otimes_X B$ and we assume that $-\otimes-$ is exact in both arguments. We also assume that
  the functor $\Perf(X)\rightarrow \tilde{\Escr}:A\mapsto  A \otimes_{X}  B$ has a right adjoint $\tilde{\Escr}\r \Dscr(X)$ which we denote by $\uRHom_{\tilde{\Escr}/X}(B,-)$. I.e. for $C\in \tilde{\Escr}$ we have functorial isomorphisms
  \[
    \Hom_{\tilde{\Escr}}(A\otimes_{X} B,C)\cong \Hom_X(A,\uRHom_{\tilde{\Escr}/X}(B,C)).
  \]
  An autoequivalence $S_{\tilde{\Escr}/X}:\tilde{\Escr}\rightarrow \tilde{\Escr}$ is said to be a \emph{relative Serre functor} for $\tilde{\Escr}/X$ if
  there are functorial isomorphisms
  \[
 \uRHom_X(   \uRHom_{\tilde{\Escr}/X}(B,C),\Oscr_X)\cong\uRHom_{\tilde{\Escr}/X}(C,S_{\tilde{\Escr}/X} B)
\]
for $B,C\in \tilde{\Escr}$. We say that $\tilde{\Escr}/X$ is \emph{strongly crepant} if the identity functor $\tilde{\Escr}\rightarrow \tilde{\Escr}$ is a relative Serre functor.

A \emph{strongly crepant categorical resolution} of $X$ is a quadruple  $(\tilde{\Dscr},\pi_\ast,\pi^\ast,\otimes_X)$  such that $(\tilde{\Dscr},\pi_\ast,\pi^\ast)$
is a categorical resolution of $X$, $-\otimes_X-$ is a $\Perf(X)$-module structure on~$\tilde{\Dscr}$ such that $\tilde{\Dscr}/X$ is strongly crepant
and $\pi^\ast$ is
$\otimes_X$-linear. The last condition means that for $A,B\in \Perf(X)$ we have functorial isomorphisms $A\otimes_X \pi^\ast B\cong \pi^\ast(A\otimes_X B)$ satisfying the appropriate compatibilities.

It is shown in \cite[\S3]{Kuz} that a strongly crepant categorical resolution is weakly crepant, and moreover that if $\pi:Y\rightarrow X$ is a crepant resolution in the usual sense
then $(\Dscr(Y),R\pi_\ast,L\pi^\ast, L\pi^\ast(-)\otimes_Y -)$ is a strongly crepant categorical resolution of $X$.

The following easy lemma, which is an extension of \cite[Example 5.3]{Lunts}, shows that, under suitable conditions, rings of the form $\End_R(M)$
form crepant categorical resolutions. If $M$ is an $R$-module then $\add(M)$ is the category spanned by modules which are direct summands of some $M^{\oplus n}$.
\begin{lemmas}
  \label{ref-1.1.2-2} Assume that $X=\Spec R$ is an algebraic variety and let $M$ be a finitely generated  $R$-module such that $\Lambda=\End_R(M)$ has finite global dimension. Then $\Dscr(\Lambda)$ is smooth. Assume in addition that $R\in \add(M)$. Then
\begin{equation}
  \label{eq:catret}
\Perf(R)\rightarrow \Dscr(\Lambda): N\mapsto M\Lotimes_R N
\end{equation}
yields a categorical resolution of singularities of $X$ (since $\Perf(R)\cong \Perf(X)$). Moreover, assuming
furthermore that $R$ is normal Gorenstein:
\begin{enumerate}
\item if  $M$ is maximal Cohen-Macaulay then this categorical resolution is weakly crepant;
\item if $\Lambda$ is an NCCR then this categorical resolution is strongly crepant.
\end{enumerate}
\end{lemmas}
Note that if (2) holds then $M$ is maximal Cohen-Macaulay since we have assumed that $R\in \add(M)$.

The hypotheses of Lemma \ref{ref-1.1.2-2} are actually too strong. For example an NCCR is  always a strongly crepant categorical
resolution in dimension $\le 3$. This follows from
Proposition \ref{prop:sufficient} below
which can be proved using the methods of \cite{IR,IyamaWemyssNCBO}.
\begin{proposition} \label{prop:sufficient}
  Assume that $\Lambda=\End_R(M)$ is an NCCR and
$\dim R\le 3$ then
\begin{equation}
  \label{eq:extvanishing}
  \Ext^i_\Lambda(M,M)=0\qquad \text{for $i>0$}.
\end{equation}
\end{proposition}
\begin{proof} For the benefit of the reader we give a proof. 
We may assume that $R$ is local of dimension 3 (the case
$\dim\le 2$ is easy). By the Auslander Buchsbaum formula \cite[Proposition 2.3]{IR} $\Lambda$ has global dimension $3$.
Since $M$ is reflexive, it has depth $\ge 2$, and hence, again by the Auslander-Buchsbaum formula, it has projective dimension $\le 1$ over $\Lambda$ and moreover
it is projective over $\Lambda$ in codimension $2$.

Hence we
have a projective resolution of $M$ as $\Lambda$-module:
\[
  0\rightarrow P_1\rightarrow P_0\rightarrow M\rightarrow 0.
\]
Applying $\Hom_A(-,M)$ we get a long exact sequence of $R$-modules
\begin{equation}
  \label{eq:depth}
  0\rightarrow R \rightarrow \Hom_\Lambda(P_0,M)\rightarrow \Hom_\Lambda(P_1,M)\rightarrow \Ext^1_\Lambda(M,M)\rightarrow 0.
\end{equation}
Assume $\Ext^1_\Lambda(M,M)\neq 0$.
Since $M$ is projective over $\Lambda$ in codimension two, $\Ext^1_\Lambda(M,M)$ is finite
dimensional and hence it has depth 0 as $R$-module. On the other hand, since $\Hom_\Lambda(P_i,M)$ is reflexive as $R$-module, it has $\depth \ge 2$.
Finally~$R$ being maximal Cohen-Macaulay has depth 3. One may verify that these depth restrictions  are incompatible with
\eqref{eq:depth}.
\end{proof}
It seems too much to hope for that \eqref{eq:extvanishing} would always be true, but lack of time has prevented us from seriously looking for a counterexample.
On the other hand we are sufficiently optimistic to make the following conjecture.
\begin{conjecture} If $X=\Spec R$ is a normal algebraic variety with Gorenstein singularities then an NCCR of $R$ always yields a strongly crepant categorical resolution of $X$.
\end{conjecture}
To prove this conjecture one would have to construct for an NCCR $\Lambda$ of $R$ a partial tiling complex $P^\bullet$ of $\Lambda$-modules such that
$\RHom_\Lambda(P^\bullet,P^\bullet)=R$.
\begin{remark} \label{rem:strongly} The strongly crepantness of $\Escr/X$ as defined above is independent of the resolution property. One may check that if $\Lambda/R$
  is a twisted NCCR then $\Dscr(\Lambda)$ is strongly crepant over $\Spec R$. But one may check that it is not a categorical resolution.
\end{remark}
\section{Constructions of non-commutative crepant resolutions}
\subsection{Quotient singularities} \label{ssec:qs}
Here we will restrict ourselves to quotient singularities for finite groups.
Quotient singularities for (non-finite) reductive groups will be covered in \S\ref{sec:quotient}.

If $G$ is a finite group and $W$ is a faithful finite dimensional unimodular (i.e.\ $\det W=k$) representation of $G$ then the skew group ring
$\Sym(W)\# G=\End_{\Sym(W)^G}(\Sym(W))$ is
an NCCR for $R=\Sym(W)^G$ (which is Gorenstein because of the unimodularity hypothesis).

In dimension $\le 3$ such quotient singularities always have a crepant resolution by the celebrated BKR-theorem \cite{BKR}. In higher dimension
this is not so. The simplest counterexample is given by $\ZZ_2$ acting with weights $(-1,-1,-1,-1)$ on $W=k^4$ because in that case
$R$ is $\QQ$-factorial and terminal. See e.g.\ \cite{MOdiscussion}.
\subsection{Crepant resolutions with tilting complexes}
\label{ssec:tilting}
In case $R$ is a normal Gorenstein domain and $Y\rightarrow \Spec R$ is a crepant resolution and $\Tscr$ is a tilting complex on $Y$
then $\uEnd_Y(\Tscr)$ is an NCCR of $R$ by Theorem \ref{th:iw}. Conversely, assuming a crepant resolution exists, any NCCR has to be of this form if we accept Conjecture \ref{ref-4.6-10} ($\Tscr$ is the dual of the image of $\Lambda$ under the asserted derived equivalence $\Dscr(\Lambda)\cong \Dscr(Y)$).

This is a very general method for constructing NCCRs. Note however that even in dimension three there may be crepant resolutions without tilting complex.
See Example \ref{ex:dao}. Furthermore as indicated in \S\ref{ssec:qs} there are normal Gorenstein singularities that admit an NCCR but not a crepant resolution.

\begin{example}
  \label{ex:textbook}
  A textbook example where this method works very well is the case of determinantal varieties \cite{BLVdB, VdB100}.
  Let $n\ge 1$ and $0\le l< n$.
Let $X_{l,n}=\Spec R$ be the varieties of matrices in $\Hom_k(k^n,k^n)=M_{n\times n}(k)$ which have rank $\le l$.
It is a classical result that~$X$ is Gorenstein. It is also well known that $X$ has a crepant Springer type resolution given by
\[
  Y=\{(\phi,V)\mid V\in \Gr(l,n), \phi\in \Hom_k(k^n,V)\}
  \]
  where $\pi:Y\rightarrow X$ sends $(\phi,V)$ to the composition of $\phi$ with the inclusion $V\hookrightarrow k^n$. If $\Rscr$ denotes the universal subbundle
  on $\Gr(l,n)$ then $Y$ is the vector bundle $\Hom(k^n,\Rscr)$ (i.e.\ $Y=\underline{\Spec} \Sym((\Rscr^{\oplus n})^\vee)$). Using Bott's theorem one computes that the Kapranov tilting bundle on $\Gr(l,n)$ \cite{Kapranov3}
  (see also \cite{BLV1000}\cite{MR3558208} for the case of finite characteristic), pulls back to a tilting bundle on $Y$, which then gives an NCCR of $R$.
  For other approaches to this example see \cite{SegalDonovan} and Theorem \ref{ref-1.4.3-14} below. 
\end{example}
Alas things are often more complicated. For determinantal varieties associated to symmetric of skew-symmetric matrices,
the Springer type resolutions  are not crepant so a tilting bundle on them only gives an NCR (see \cite{WeymanZhao}). NCCRs of such generalized determinantal
varieties will be obtained in \S\ref{sec:quotient} using a different approach.
\begin{example}
 Another beautiful, and much deeper  example \cite{Bez,BMR} is given by cotangent bundles of (partial) flag
 varieties $T^\ast(G/P)$. If $P$ is a Borel subgroup of $G$ then this is a crepant resolution of the nilpotent cone in $\Lie(G)$.
In general 
 they are crepant resolutions of closures of Richardson orbits \cite{NamikawaRichardson}.  It is shown in \cite{Bez,BMR} that
 $T^\ast(G/P)$ has a tilting bundle but it is not obtained as the pullback of a tilting bundle on $G/P$. In fact the construction of the tilting bundle
 is highly non-trivial. 
 To explain the construction it is useful to exhibit
 a slightly different point of view on tilting bundles.

 Let $Y$ be a noetherian scheme. If $\Ascr$ is a quasi-coherent sheaf of algebras on $Y$ 
  and $A=\Gamma(Y,\Ascr)$ then we say that $\Ascr$ is \emph{derived affine}
  if $A=R\Gamma(Y,\Ascr)$ and the right orthogonal to $\Ascr$ in
  $D_{\Qch}(\Ascr)$ is zero. In that case
  $
    R\Gamma(Y,-)
   $ defines an equivalence of categories between $D_{\Qch}(\Ascr)$ and $D(A)$.
  It is not difficult to see that
  a vector bundle $\Tscr$ on $Y$ which has everywhere non-zero rank is
  a tilting bundle provided  $\uEnd_Y(\Tscr)$ is derived affine.\footnote{To verify this it is useful to observe that $\Tscr$ is tilting if and only
    if $\Tscr^\vee$ is tilting. The only non-trivial part is the generation property. To this end one may use that $\Tscr$ generates $D_{\Qch}(Y)$ if and only
  if $\Perf(Y)$ is the smallest \'epaisse subcategory of $D_{\Qch}(Y)$ containing $\Tscr$ \cite[Lemma 2.2]{Neeman4}, together with the fact that $(-)^\vee$ is an autoequivalence of $\Perf(Y)$.} 

\medskip

  We say that $Y$ is derived $\Dscr$-affine if $\Dscr_Y$ is derived affine
  where $\Dscr_Y$ is the sheaf of differential operators on $Y$. In characteristic
  $>0$ we mean by $\Dscr_Y$ the sheaf of \emph{crystalline} differential
  operators, i.e.\ differential operators which may be expressed in terms
  of derivations, without using divided powers.

  \medskip

  Now let $Z=G/P$. The Bernstein Beilinson theorem \cite{BB}, valid in characteristic zero, states that $Z$ is even ``$\Dscr$-affine''
  meaning that the equivalence $R\Gamma(Z,-)$ is also compatible with the natural t-structures. This is
  false in characteristic $>0$. However $Z$ is still derived~$\Dscr$-affine \cite[Theorem 3.2]{BMR} whenever $p$ is strictly bigger than the Coxeter number, which we will assume now.

  We will give a rough sketch how this is used in \cite{Bez,BMR} to construct a tilting bundle on $Y=T^\ast Z$.
  Let us first assume that the characteristic of $k$ is $p>0$. To indicate this we will adorn our notations with $(-)_p$.
  In that case $\Dscr_{Z_p}$ is coherent as a module over its center which is equal to $(\operatorname{Sym}_{Z_p}\Escr_p)^{(1)}$ where $(-)^{(1)}$ denotes
  the Frobenius twist, and $\Escr_p$ is the tangent bundle on $Z_p$. Hence we may view $\Dscr_{Z_p}$ as a sheaf of
  coherent algebras
  $\tilde{\Dscr}$ on $\underline{\operatorname{Spec}}(\operatorname{Sym}_{Z_p}\Escr_p)^{(1)}=Y_p^{(1)}$ where $Y_p=T^*Z_p$. The sheaf $\tilde{\Dscr}$ is still derived
  affine.

  Now $\tilde{\Dscr}$ is not of the form $\uEnd_{Y^{(1)}_p}(\Tscr^{(1)}_p)$. However if we let $\hat{Y}_p$ be the formal completion of $Y_p$ at the zero section 
then  it turns out that the restriction $\hat{\Dscr}$   of $\tilde{\Dscr}_p$ to
  to $\hat{Y}^{(1)}_p$ is of the form $\uEnd_{\hat{Y}_p}(\hat{\Tscr}_p)^{(1)}$ for a vector bundle $\hat{\Tscr}_p$ on $\hat{Y}_p$. Moreover $\hat{\Dscr}$
  is still derived affine and so $\hat{\Tscr}_p$ is a tilting bundle on $\hat{Y}_p$.
  Then one uses the rigidity\footnote{Tilting bundles have in particular vanishing $\Ext^{1,2}$.
    Hence by classical deformation theory they are unobstructed and rigid. }
    of tilting bundles
  to lift $\hat{\Tscr}_p$ to a tilting bundle $\hat{\Tscr}$ in characteristic zero. Finally one may use
  the fact that $Y=T^\ast Z$ (as a vector bundle) admits a nice $G_m$ action to conclude by \cite[Theorem 1.8]{Kaledin}
  that $\hat{\Tscr}$ is actually the completion of a tilting bundle $\Tscr$ on
  $Y$.
\end{example}
  Hidden behind this construction is the fact that $\Dscr_Z$ is in some sense a canonical non-commutative deformation of the symplectic variety $T^\ast Z$.
  If $Y$ is a general symplectic variety then one may try to construct a non-commutative deformation using Fedosov quantization. This general idea
  has been used by Bezrukavnikov and Kaledin to prove an analogue of the BKR theorem \cite{BKR} for crepant resolutions of symplectic quotient singularities \cite{BezKal} and by Kaledin
  to prove a suitable version of Conjecture \ref{con:BO} \cite{Kaledin} for general symplectic singularities. To apply the method one needs to be able to do Fedosov quantization in finite
  characteristic, a problem which has been solved to some extent in \cite{BezKal2}.

  \subsection{Resolutions with partial tilting complexes}
  \label{ssec:partial}
  Assume $R$ is a normal Gorenstein domain with rational singularities and $Y\rightarrow \Spec R$ is a
  resolution which is not crepant. A strengthening of Conjecture \ref{ref-4.6-10} inspired by \cite{Kuz} is that NCCRs are minimal in a categorical sense, i.e.\ their derived category embeds
  inside~$\Dscr(Y)$.  This means that they are obtained 
   as $\uEnd_Y(\Tscr)$ for a
  partial tilting complex $\Tscr$ on $Y$. For a very general
  result in this direction see \cite[Theorem 2]{Kuz}.  We will
  restrict ourselves to a special case which will be useful in
  \S\ref{sec:stringy} and which can be easily proved directly.

\begin{proposition}[\protect{\cite{Kuz}}]   \label{prop:Kuz}
Let $Z$ be a smooth projective variety with ample line bundle $\Oscr_Z(1)$ and let $X=\Spec R$ be the corresponding cone.
Assume $\omega_Z=\Oscr_Z(-n)$ for $n\ge 1$.  Then $R$ is Gorenstein. Moreover a resolution of singularities $\pi:Y\rightarrow X$ of $X$ is given by
the line bundle over $Z$ associated to $\Oscr_Z(1)$.
Assume $\Escr\in \Dscr(Z)$ is such that:
\begin{enumerate}
\item \label{it:van1} $\Ext^i_Z(\Escr,\Escr(m))=0$ for $i>0$ and $m\ge 0$;
\item \label{it:van2} $\Ext^i_Z(\Escr,\Escr(m))=0$ for $i\ge 0$ and $m\in \{-1,\ldots, -n+1\}$;
\item \label{it:van3} $\Escr\oplus \Escr(1)\oplus \ldots \oplus \Escr(n-1)$ is a generator for $D_{\Qch}(Z)$.
\end{enumerate}
Let $\gamma:Y\rightarrow Z$ be the projection map and put $\Tscr=\gamma^\ast \Escr$. Then $\End_Y(\Tscr)$ is an NCCR of $R$.
\end{proposition}
\begin{proof} We write $\Tscr(m)=\gamma^\ast(\Escr(m))$. Then we have
  \begin{equation}
    \label{eq:homcalc}
    \RHom_Y(\Tscr,\Tscr(m))=\bigoplus_{l\ge 0} \RHom_Z(\Escr,\Escr(m+l)).
  \end{equation}
  Using (\ref{it:van1},\ref{it:van2}) we deduce in particular that $\bar{\Tscr}:=\Tscr\oplus \cdots\oplus \Tscr(-n+1)$ is partial tilting
  (and hence this is also the case for $\Tscr$). Furthermore from \eqref{it:van3} we obtain $\bar{\Tscr}^\perp =0$. So
  $\bar{\Tscr}$ is in fact tilting. Put $\Lambda=\End_Y(\Tscr)$, $\bar{\Lambda}=\End_Y(\bar{\Tscr})$. By Theorem \ref{th:kellerkrause} $\bar{\Lambda}$
  has finite global dimension.
  
  Via the decomposition \eqref{eq:homcalc} $\Lambda$ is a $\NN$-graded ring. Put $\Lambda_{\ge u}=\bigoplus_{m\ge u} \Lambda_m$. Then
  (as ungraded rings) we have
  \[
\bar{\Lambda}=
    \begin{pmatrix}
      \Lambda & \Lambda_{\ge 1} &\cdots &\Lambda_{\ge n-1}\\
      \Lambda & \Lambda &\cdots &\Lambda_{\ge n-2}\\
      \vdots & \vdots &\ddots &\vdots\\
      \Lambda & \Lambda &\cdots &\Lambda_{\ge 1}\\
      \Lambda & \Lambda &\cdots &\Lambda
      \end{pmatrix}.
    \]
    If we put $\Gamma=M_n(\Lambda)$ then $\bar{\Lambda}\subset \Gamma$ and moreover $\Gamma$ is (left and right) projective over~$\bar{\Lambda}$ and in addition the  multiplication map $\Gamma\otimes_{\bar{\Lambda}}\Gamma\rightarrow \Gamma$ is an isomorphism (it is a surjective map between projective $\Gamma$-modules of the same rank). We claim that $\Gamma$ (and hence $\Lambda$) has finite global dimension. Indeed if $M$ is a right $\Gamma$-module and $P^\bullet \rightarrow M$ is a
    finite projective resolution of $M$ as $\bar{\Lambda}$-module (which exists since $\gldim \bar{\Lambda}< \infty$) then
    $\Gamma\otimes_{\bar{\Lambda}} P^\bullet$ is a finite $\Gamma$-projective resolution of $\Gamma\otimes_{\bar{\Lambda}} M =
 \Gamma\otimes_{\bar{\Lambda}} \Gamma \otimes_{\Gamma}   M\cong M$.

  Moreover for $i>0$:
  \begin{align*}
    \Ext^i_R(\End_Y(\Tscr),\omega_R)&=\Ext^i_X(\pi_\ast \uEnd_Y(\Tscr),\omega_X)\\
    &=    \Ext^i_Y(\uEnd_Y(\Tscr),\omega_Y)\\
                                              &=\Ext^i_Y(\Tscr,\Tscr(-n+1))\\
    &=0
  \end{align*}
where in the second line we have used Grothendieck duality, in the third line the easily verified fact that $\omega_Y=\gamma^\ast(\omega_Z(1))$ 
and in the fourth line \eqref{eq:homcalc} and (\ref{it:van1},\ref{it:van2}). It follows
that $\End_Y(\Tscr)$ is maximal Cohen-Macaulay over $R$.
\end{proof}
\subsection{Three-dimensional affine toric varieties}
\label{sec:toric}
For simplicity we define an affine toric variety as $X=\Spec R$ where $R=k[\sigma^\vee \cap M]$ where $M$ is a lattice and
$\sigma^\vee$ is a strongly convex full dimensional
lattice cone in $M_\RR$. Such an $R$ is Gorenstein if there exist $m\in M$ such that $\sigma$ (the dual cone of $\sigma^\vee$) is spanned by
lattice vectors $x\in M^\vee$ satisfying $\langle x, m\rangle=1$. The lattice polytope associated to $R$ is defined as $P =\sigma \cap \langle m,-\rangle$.

In this case there is the following beautiful result by Broomhead \cite[Theorem 8.6]{Broom}.
\begin{theorem} \label{th:broomhead}
  The coordinate ring of a 3-dimensional Gorenstein affine toric variety admits a toric NCCR.
\end{theorem}
By a toric NCCR we mean that the reflexive module defining the NCCR is
isomorphic to a sum of  ideals.  Broomhead's proof uses the
theory of ``dimer models'' which is possible thanks to the combinatorics 
\cite{Gulotta,IshiiUeda}.
A proof not using dimer models but using
this combinatorics directly was given in
\cite{SVdB5}.

\medskip

A different method for constructing NCCRs for affine Gorenstein toric varieties was given in \cite{SVdB4} and is based on a standard fact from toric geometry:
\begin{lemma} \label{prop:dm} A subdivision of $\sigma$ obtained by a regular triangulation of $P$ with no extra vertices
  yields a projective crepant resolution of $\Spec R$ by a toric Deligne Mumford stack \cite{borisov2005orbifold}. If $\dim X\le 3$ then such a crepant resolution   has fibers of dimension $\le 1$.
  \end{lemma}
In dimension $\le 3$ one may then, starting from a sequence
  of generating line bundles, construct a tilting bundle using the principle of ``killing backward $\Ext^1$'s'' (see Remark \ref{rem:backward}). 

\medskip  

  While this method yields an NCCR,
it generally does not yield a toric one. On the other hand it is also applicable to some higher dimensional toric singularities which do not have a toric NCCR.
\begin{example}[\protect{\cite[\S9.1]{SVdB},\cite[Example 6.4]{SVdB4}}]
  \label{ex:toric}
Let $T=G^2_m$ be the two dimensional torus and (after the identifying the character group $X(T)$ of $T$ with $ \ZZ^2$) 
consider the vector space  $W$
with weights  $(3,0)$, $(1,1)$, $(0,3)$, $(-1,0)$, $(-3,-3)$, $(0,-1)$. Put $R:=\operatorname{Sym}(W)^T= k[x_1,x_2,x_3,x_4,x_5,x_6]^T=k[x_2x_4x_6,x_1x_3x_5,x_1x_4^3,x_3x_6^3,x_2^3x_5]\cong k[a, b, c, d, e]/(a^3 b -cde)$. Clearly $R$ is the coordinate ring of a 4-dimensional affine toric variety, but it was  shown in
\cite[\S9.1]{SVdB} that $R$ does not have a toric NCCR.

On the other by \cite[Proposition 6.1]{SVdB4}, $R$ 
does have a non-toric NCCR. In \cite[Example 6.4]{SVdB4} an explicit NCCR is constructed which is given by a reflexive module which is the direct
sum of 12 modules of rank 1 and 1 module of rank 2.
\end{example}
We conjecture:
\begin{conjecture} \label{conj:nccrtoric}
  An affine Gorenstein toric variety always has an NCCR.
\end{conjecture}
Besides Theorem \ref{th:broomhead} this conjecture is also true for ``quasi-symmetric GIT quotients'' for tori. See Corollary \ref{cor:162}  below.

By \cite[Theorem A.1]{SVdB4} the Grothendieck group of the DM-stack exhibited in Lemma \ref{prop:dm}
has rank $\Vol(P)$.
This suggests the following conjecture:
\begin{conjecture}
  \label{conj:volume}
The number of indecomposable summands in the reflexive module defining an NCCR of $R$ is equal to $\Vol(P)$.
\end{conjecture}

\subsection{Mutations}
It follows from the minimal model program that the number of crepant resolutions of an algebraic variety is
finite\footnote{I thank Shinnosuke Okawa for explaining to
  me how this follows from \cite{BCHM}.}. On the other hand NCCRs can
be modified by a process called ``mutation'' which is closely related to flopping of crepant resolutions. The difference is
that the mutation process generally leads 
to an infinite number of different NCCRs (see however Example \ref{ex:IW2} below).

The following definitions and results are taken from \cite{IW}. Let $R$ be a normal Gorenstein ring.
Let $M$ be a reflexive $R$-module such that $\Lambda=\End_R(M)$ is an NCCR and let $0\neq N\in \add(M)$.
Let $K_0$ be defined by the short exact sequence
\[
  0\rightarrow K_0\rightarrow N_0\rightarrow M
\]
where $N_0\in \add(N)$  is a \emph{right approximation} of $M$, i.e. any other map $N'_0\rightarrow M$ with $N'_0\in \add(N)$ factors through $N_0$.
One defines the \emph{right mutation} of $M$ at $N$ to be $\mu^+_N(M):=N\oplus K_0$. The \emph{left mutation} of $M$ at $N$ is defined
via duality as $\mu^{-}_N(M)=(\mu^+_{N^\vee}(M^\vee))^\vee$. We also put $\mu^{\pm}_N(\Lambda)=\End_R(\mu^\pm_N(M))$. Note however that
the passage from $\mu^{\pm}_N(M)$ to $\mu^{\pm}_N(\Lambda)$ loses some information.

\begin{remark}
Needless to say that $\mu^{\pm}_N(M)$ is only determined up to additive closure (i.e.\ up to taking $\add(-)$).
However if $R$ is complete local then we can make
a minimal choice for $\mu^+_N(M)$ which we will do silently.
\end{remark} 
\begin{theorem}[\protect{\cite[Theorems 1.22, 1.23]{IW}}]
Let $M,N, \Lambda$ be as above.
\begin{enumerate}
\item $\mu^{\pm}_N (M)$ define NCCRs.
  \item $\Lambda$, $\mu^+_N(\Lambda)$ and $\mu^-_N(\Lambda)$ are all derived equivalent.
  \item $\mu^+_N$ and $\mu^-_N$ are mutually inverse operations (this statement makes sense since $N\in \add(\mu^{\pm}_N(M))$). 
  \end{enumerate}
\end{theorem}
If $R$ is complete local of dimension 3, things simplify.
Let us call a reflexive $R$-module \emph{basic} if every indecomposable summand appears
occurs only once.
\begin{theorem}[\protect{\cite[Theorems 1.25]{IW}}]
  Assume that $R$ is complete local of dimension~3. Let $M$ be a basic reflexive $R$-module defining an NCCR, having at least two non-isomorphic
  indecomposable summands and let $M_i$ be such an indecomposable summand.
  Then $\mu_{M/M_i}^+(M)\cong \mu_{M/M_i}^-(M)$.
\end{theorem}
\begin{example}[\protect{\cite{HiranoWemyss, IWWeb, WemyssLockdown}}] \label{ex:IW2}
If $R$ is complete local ring with a 3-dimensional terminal Gorenstein singularity then the 
basic reflexive modules  yielding an NCCR correspond to the maximal
cells in an affine hyperplane arrangement of dimension $\rk \operatorname{Cl}(R)$ with mutations at indecomposable summands corresponding to wall crossings \cite[Theorem 4.4]{WemyssLockdown}. The group $\operatorname{Cl(R)}$ acts by translation on this hyperplane arrangement and the quotient consists of a finite number
of cells which correspond to the NCCRs of $R$. The number of such NCCRs is generally higher than the number of crepant resolutions.

It is an interesting problem to understand this for other types of 3-dimensional singularities.
\end{example}
\begin{remark}
  If $\Lambda=\End_R(M)$ is an NCCR then because of the reflexive Morita equivalence $\Refe(\Lambda)=\Refe(R)$ the mutation procedure
  may also be defined on the level of reflexive $\Lambda$-modules (see \cite[\S5]{IR}). The resulting procedure also works for twisted NCCRs, where there is no reflexive Morita equivalence.
\end{remark}
We now describe a different point of view on mutations, taken from \cite{DWZ}.
For $Q$  a quiver with $n$ vertices let $\widehat{kQ}$ be the completion of the path algebra of $Q$ at path length.
A \emph{potential} $w\in \widehat{kQ}$ is a convergent sum of cycles considered up to rotation (or equivalently $w\in \widehat{kQ}/[\widehat{kQ},
\widehat{kQ}]\,\hat{}$). If $w$ is a potential then  $(\partial w)$ denotes the (completed) two sided ideal generated by the cyclic derivatives $\partial_x w$
of $w$ with respect to the arrows in $Q$, where for a cyclic path $m$ we have $\partial_x m:=\sum_{m=uxv}vu$ (note that this is invariant under path rotation). The \emph{completed Jacobi algebra}
associated to $(Q,w)$ is defined as $\hat{J}(Q,w):=\widehat{k Q}/(\partial w)$. We say that $w$ is \emph{reduced} if it only contains cycles of length $\ge 3$.
We can also consider the uncompleted version $J(Q,w):=k Q/(\partial w)$, in case $w$ is a finite sum. We have the following result.
\begin{theorem}[\protect{\cite[Theorems A\&B]{MR3338683}}] If $\Lambda$ is a basic (i.e.\ $\Lambda/\rad\Lambda\cong k^{\oplus n}$) twisted NCCR
  of a 3-dimensional normal Gorenstein complete local ring
then
  $\Lambda$ is a completed Jacobi algebra $\hat{J}(Q,w)$ with $w$ reduced.
\end{theorem}
If $Q$ does not have loops or $2$-cycles then the mutations of $\Lambda:=\hat{J}(Q,w)$ can be obtained
by an alternative procedure described in~\cite{DWZ}. 
The procedure to mutate at a vertex $i$ of $Q$ yields a new Jacobi algebra $\hat{J}(Q',w')$ defined as follows
(see \cite[\S2.4]{KY}).
\begin{enumerate}
  \item  For each arrow $\beta$ with target $i$ and each arrow $\alpha$ with source $i$, add a new
    arrow $[\alpha\beta]$ from the source of $\beta$ to the target of $\alpha$.
    \item
 Replace each arrow $\alpha$ with source or target $i$ with an arrow $\alpha^\ast$ in the opposite
direction.
\end{enumerate}
The new potential $w'$ is the sum of two potentials $w_1^\prime$ and $w_2^\prime$ . The potential $w_1^\prime$
is obtained from $w$ by replacing each composition $\alpha\beta$ (up to cyclic rotation) by~$[\alpha \beta]$, where $\beta$ is an arrow
with target $i$. The potential $w_2'$  is given by
\[
w_2^\prime =
\sum_{\alpha,\beta}
[\alpha \beta]\beta^\ast \alpha^\ast
\]
where the sum ranges over all pairs of arrows $\alpha$ and $\beta$ such that $\beta$ ends at $i$ and $\alpha$
starts at $i$.  It follows from \cite[Theorem 3.2]{KY} that this mutation coincides with the mutation defined in \cite{IW} and described above.

\medskip

It may be that $w'$ is not reduced, i.e.\ it contains $2$-cycles. In that case the corresponding relations allow one to eliminate some arrows in $Q'$.
By doing this we find that the Jacobi algebra $\hat{J}(Q',w')$ can be more economically written as
$J((Q')^{\text{red}},(w')^{\text{red}})$ where~$(w')^{\text{red}}$ is reduced.

If we are lucky that $(Q')^{\text{red}}$ does not contain any 2-cycles (it cannot contain loops) then we can repeat the mutation procedure at arbitrary
vertices. If we can
keep doing this forever then we call the original potential $w$ \emph{non-degenerate}.

Note that if $(Q')^{\text{red}}$ does not contain 2-cycles, it can be
obtained from $Q'$ by simply deleting all $2$-cycles, so that the
mutation procedure becomes to some extent combinatorial \cite{KellerMutations}. For a
non-degenerate potential this nice property persists under iterated
mutations. The catch however is that in general it is not clear 
how to check that a potential is non-degenerate. A useful criterion, based on the theory of graded mutations \cite{MR3291645}, is given in \cite{VdBdT2}.
\begin{theorem}[\protect{\cite[Corollary 1.3]{VdBdT2}}] \label{th:graded}
  Assume:
  \begin{enumerate}
    \item
      $Q$ is a $\ZZ$-graded quiver such that $(kQ)_{\le 0}$ is finite dimensional.
    \item $Q$ has at least three vertices.
    \item  $w$ is a homogeneous reduced potential of degree $r$ (in particular
      it is a finite sum).
    \item $\Lambda=J(Q,w)$ is a twisted NCCR whose center is 3-dimensional with an isolated singularity.
    \item \label{it:automatic} $\Lambda/[\Lambda,\Lambda]$ does not contain any elements whose degree is in the interval $[1,r/2]$.
    \end{enumerate}
    Then $w$ is non-degenerate.
  \end{theorem}
Note that \eqref{it:automatic} is automatic if $r=1$. This gives an alternative proof why the potentials associated to ``rolled up helix algebras''
of Del Pezzo surfaces are non-degenerate (see \cite[Theorem 1.7]{BridgelandStern}\cite[Theorem 4.2.1]{VdBdT2}).
Theorem \ref{th:graded} also applies to many skew group rings $\Lambda=k[x,y,z]\#(\ZZ/n\ZZ)$. For example $n=5$ and $\bar{1}$ acting with weights $(1/5,2/5,2/5)$
(see \cite[\S7]{IR}).

\section{Quotient singularities for reductive groups}
\label{sec:quotient}
\subsection{NCCRs via modules of covariants}
\label{ssec:covariants}
In this section we discuss some results from \cite{SVdB}.
$G$ will always be a reductive group. Let $S$ be the coordinate ring of a smooth affine $G$-variety $X$. Then $S^G$ is the coordinate ring of the categorical
quotient $X\quot G$.
We will be interested
in constructing (twisted) NCCRs for $S^G$. In the particular case that $G$ is finite and $X$ is  a faithful unimodular $G$-representation
then this was discussed in \S\ref{ssec:qs}. An NCCR for $S^G$ is given by the skew group ring $\Lambda=S\# G$.
However $\Lambda$ can be described in a different way. For $U$ a finite dimensional $G$-representation put $M(U):=(U\otimes S)^G$.
Then $M(U)$ is a reflexive $S^G$ module (in fact it is maximal Cohen-Macaulay). If every irreducible representation of $G$ occurs at least once in $U$ then
$\Lambda$ is Morita equivalent to $\End_{S^G}(M(U))$. Hence $M(U)$ defines an NCCR of $S^G$.

\medskip

The modules $M(U)$ we introduced are so-called \emph{modules of
  covariants} \cite{Brion2} and they make perfect sense for general reductive groups.
A mild obstacle is that modules of covariants do
not have to be reflexive in general \cite{Brion2}.
This is not a serious problem, but if we want to avoid it anyway, we can
restrict the pairs $(G,X)$ we consider. We will say that $G$ acts
\emph{generically} on a smooth affine variety if the locus
of  points with closed orbit and trivial stabilizer is non-empty and its complement
has codimension $\ge 2$.
If $W$ is a $G$-representation then we will say that $(G,W)$ is \emph{generic} if $G$ acts generically on
$\Spec \Sym W\cong W^\ast$.
We then have
in particular 
\begin{equation}
  \label{eq:modcov}
\End_{S^G}(M(U))=M(\End(U)).
\end{equation}

It is reasonable to search for NC(C)Rs of the form
$\End_{S^G}(M(U))$. However if~$G$
is not finite there are non-trivial obstacles:
\begin{enumerate}
\item There are an infinite number irreducible representations so we cannot just take the
  sum of all of them. We need to make a careful selection.
\item Modules of covariants are usually not Cohen-Macaulay and so demanding that $\End_{S^G}(M(U))=M(\End(U))$
(cfr \eqref{eq:modcov}) is Cohen-Macaualay is a severe restriction on $U$.
\end{enumerate}
The first issue is handled in \cite[\S10]{SVdB} where we construct certain nice complexes 
relating different modules of covariants (see also \cite[Chapter 5]{weyman2003cohomology}). 
The second
issue is handled using results from \cite{VdB3} (see also \cite{Vdb1,Vdb7,VdB9,VdB4}).

\medskip

Before we discuss NCCRs let's give a result on NCRs.
\begin{proposition}[\protect{\cite[Corollary 1.3.5]{SVdB}}]
Assume that $(G,W)$ is generic.  Then there exists a finite
dimensional $G$-representation $U$ containing the trivial
representation such that $\Lambda=\End_{S^G}(M(U))$ is an NCR for
$S^G$.
\end{proposition}
\begin{remark} The fact that $U$ contains the trivial representation implies that $\Lambda$ defines a categorical resolution by Lemma \ref{ref-1.1.2-2}.
It turns out that NCRs are easier to construct than NCCRs since it is
sufficient to
take $U$ big enough, in a suitable sense.
\end{remark}
To state our results about (twisted) NCCRs we need to introduce some notation.
Let $G$ be a connected\footnote{In \cite{SVdB} we also consider the non-connected case.} reductive group. Let
$T\subset B\subset G$ be respectively a maximal torus and a Borel
subgroup of $G$ with $\Wscr=N(T)/T$ being the corresponding Weyl group. Put
$X(T)=\Hom(T,G_m)$ and let $\Phi\subset X(T)$ be the roots of $G$.  
By convention
the roots of~$B$ are the negative roots $\Phi^-$ and $\Phi^+=\Phi-\Phi^-$ is the set
of  positive roots. 
We write $\bar{\rho}\in X(T)_\RR$ for half the sum of the positive roots.
 Let $X(T)^+_\RR$ be the dominant cone in $X(T)_\RR$ and let $X(T)^+=X(T)^+_\RR\cap X(T)$ be the set of dominant
weights. For $\chi\in X(T)^+$ we denote the simple
$G$-representation with highest weight~$\chi$ by~$V(\chi)$.

Let $W$ be a finite dimensional $G$-representation of dimension $d$ and put $S=\Sym(W)$, $X=\Spec \Sym(W)=W^\ast$.
Let 
$(\beta_i)_{i=1}^d\in X(T)$ be the $T$-weights of $W$. 

 Put
\begin{align*}
\Sigma&=\left\{\sum_i a_i \beta_i\mid a_i\in ]-1,0]\right\}\subset X(T)_\RR.
\end{align*}
The elements of the intersection
$
X(T)^+\cap(-2\bar{\rho}+\Sigma)
$ 
are called \emph{strongly critical (dominant) weights} for $G$.
\begin{theorem}[\protect{\cite[Theorem 3.4.3]{SVdB}\cite{VdB3}}] \label{th:strongly}
  Assume that $X$ contains a point with closed orbit and finite stabilizer. Let $\chi\in X(T)^+$ be a strongly critical weight and $U=V(\chi)$. Then 
$M(U^\ast)$ is a Cohen-Macaulay $S^G$-module.
\end{theorem}
If we look at \eqref{eq:modcov} and observe that  the weights of $\End(U)$  are very roughly speaking about twice those of $U$, then Theorem \ref{th:strongly} suggests
that to
construct an NCCR we should restrict ourselves to representations whose highest weights are approximately contained in $-\bar{\rho}+(1/2) \Sigma$. This idea
works for the class of ``quasi-symmetric'' representations, which includes the class of self dual representations.

We say that $W$
is \emph{quasi-symmetric} if for every line $\ell\subset X(T)_\RR$ through
the origin we have 
$
\sum_{\beta_i\in\ell}\beta_i=0
$.
This implies in particular that $W$ is unimodular
and hence~$S^G$ is Gorenstein if $W$ is generic by a result of Knop \cite{Knop3}. 

\medskip

From now on we assume that $(G,W)$ is generic and $W$ is quasi-symmetric. Deviating slightly from \cite{SVdB}, following \cite{HLSam}, we introduce a certain affine hyperplane arrangement
on $X(T)^{\Wscr}_\RR$. Let $\bar{\Hscr}$ be the collection of affine hyperplanes spanned by the facets of $-\bar{\rho}+(1/2)\bar{\Sigma}$. We consider
the hyperplane arrangement in\footnote{Note that $X(T)^\Wscr$ is just the character group $X(G)$ of $G$.} $X(T)^\Wscr_\RR$ given by:
\begin{equation}
  \label{eq:hyper}
  \Hscr=\bigcup_{H\in \bar{\Hscr}} (-H+X(T))\cap X(T)^\Wscr_\RR.
\end{equation}
\begin{remark} The hyperplane arrangement \eqref{eq:hyper} may be degenerate in the sense that $X(T)^\Wscr_\RR \subset -H+\chi$ for some $\chi \in X(T)$.
\end{remark}
\begin{example}
 We give a simple example where degeneration occurs. Let  $G=\Sl(2)$. If $V$ is the standard
  representation and $W=V^{n}$ with $n$ even, then $(1/2)\Sigma$ is the interval
  $]-n/2, n/2[$ (identifying $X(T)\cong \ZZ$). Moreover $\bar{\rho}=1$.
  Hence the ``hyperplanes'' $-H+X(T)$ are given by the integers. Furthermore
  $X(T)^\Wscr_\RR=0$. Thus the induced hyperplane arrangement in $X(T)^\Wscr_\RR$ is indeed degenerate. If $n$ is odd on the other hand then it is non-degenerate.
  See \cite[Theorem 1.4.5]{SVdB} for a complete treatment of the case
  $G=\Sl(2)$.
\end{example}

This hyperplane arrangement is such that if $\delta$ is the complement of $\Hscr$ then
  \[
(-\bar{\rho}+\delta+1{/}2\,\partial\bar{\Sigma})\cap X(T)=\emptyset.
\]
The following result is a slight variation on \cite[Theorem 1.6.4]{SVdB}.
\begin{theorem} \label{th:mainth}
  Let $(G,W)$ be generic and assume that $W$ is quasi-symmetric.
  Let~$\delta$ be an element of the complement of $\Hscr$. Put 
\begin{align}
    \Lscr_\delta&=X(T)^+\cap (-\bar{\rho} +\delta +(1/2)\bar{\Sigma}), \label{eq:Ldelta}\\
  U_\delta&=\bigoplus_{\chi \in \Lscr_\delta} V(\chi), \label{eq:Udelta}\\
  \Lambda_\delta &=\End_{S^G}(M(U_\delta)).
\end{align}
If $\Lscr_\delta\neq \emptyset$ then $\Lambda$ is an NCCR for $\Sym(W)^G$. 
\end{theorem}
It is easy to see that $\Lscr_\delta$ and hence $U_\delta$ depend only on the connected component of the complement of $\Hscr$ to which $\delta$ belongs.

\medskip

We obtain some evidence for Conjecture \ref{conj:nccrtoric}.\footnote{This is stated in loc.\ cit.\ for $W$ generic. However one easily reduces to this case.}
\begin{corollary}[\protect{\cite[Theorem 1.6.2]{SVdB}}]  \label{cor:162}
If $G=T$ is a torus and $W$ is quasi-symmetric then $\Sym(W)^T$ has a (toric) NCCR. 
\end{corollary}
\begin{remark}
For reference we note that there is  extension of Theorem \ref{th:mainth} that may allow one to construct twisted NCCRs \cite[Theorem 1.6.4]{SVdB}.
\end{remark}
We now state some consequences of these results for determinantal varieties.
\begin{theorem}[\protect{\cite[Theorem 1.4.1]{SVdB}}]  \label{th:det}
For $l< n$ let $X_{l,n}$ be the variety
of $n\times n$-matrices of rank $\le l$.
Then $k[X_{l,n}]$ has an NCCR.
\end{theorem}
The variety $X_{l,n}$ was already discussed in Example \ref{ex:textbook} and the NCCR obtained in loc.\ cit.\ is the same as the one we obtain in Theorem
\ref{th:det}. To prove Theorem \ref{th:det} we use the classical description of $k[X_{l,n}]$ as an invariant ring \cite{Weyl}. Put
$G=\GL(l)$ and let $V$ be the standard representation of $G$. Put $W=V^n\oplus (V^\ast)^n$. Then $k[X_{l,n}]=\Sym(W)^G$ and we show in
\cite{SVdB} that Theorem \ref{th:det} follows from Theorem \ref{th:mainth}. For the benefit of the reader we describe the actual module of covariants
that gives the NCCR.
Let $B_{l,n-l}$ be the set of partitions that fit in a rectangle of size $l\times (n-l)$.
 In loc.\ cit.\  it is shown that the following module of covariants defines an NCCR for $R$:
 \begin{equation} \label{eq:det}
   M=\bigoplus_{\lambda\in B_{l,n-l}} M(S^\lambda V),
 \end{equation}
  where $S^\lambda V$ denotes the Schur functor indexed by $\lambda$ applied to $V$.
\begin{theorem}
\label{ref-1.4.3-14} 
For $2l<n$ let $X_{2l,n}^-$ be  the variety of skew-symmetric $n\times n$ matrices of rank $\le 2l$.
If $n$ is odd then $k[X_{2l,n}^-]$ has an NCCR.
\end{theorem}
This time we put $G=\Sp(2l)$ and $W=V^n$ where $V$ is the standard representation of~$G$.
\begin{theorem}[\protect{\cite[Theorem 1.4.1]{SVdB}}]
\label{ref-1.4.5-15} For $l<n$ let $X^+_{l,n}$ be the variety of symmetric $n\times n$ matrices of rank $\le l$.
If $l$ and $n$ have opposite parity then $k[X^+_{l,n}]$ has an NCCR.
If $l$ and $n$ have the same parity then
$k[X_{l,n}^+]$ has a twisted NCCR. 
\end{theorem}
Here we put $G=\operatorname{O}(l)$ and again $W=V^n$ where $V$ is the standard representation of $G$. A complication arises since $\operatorname{O}(l)$ is not connected, so we
cannot directly apply Theorem \ref{th:mainth}. So we have to perform a more refined analysis which is carried out in \cite[\S6]{SVdB}.
Twisted NCCRs appear because $\operatorname{SO}(l)$, the connected component of $\operatorname{O}(l)$, is not simply connected.

The NCCRs given in Theorems \ref{ref-1.4.3-14}, \ref{ref-1.4.5-15} have been crucial for establishing homological projective duality \cite{KuzHPD} for
determinantal varieties of skew-symmetric matrices by Rennemo and Segal \cite{RenSeg}. The corresponding results for symmetric matrices are work in progress
by the same authors \cite{rennemo}.

\medskip

Even if $W$ is not quasi-symmetric then it is still possible that $\Sym(W)^G$ has an NCCR given by a module of covariants but we are
unaware of a general rule like Theorem~\ref{th:mainth} for constructing them. Three-dimensional
affine toric varieties (see \S\ref{sec:toric}) are an example of this, since they can be written as $\Sym(W)^G$ where $W$ is generally not
quasi-symmetric. Another example is given by 
the recent work of Doyle:
\begin{example}[\protect{\cite[Theorem 3.11]{Doyle}}]\label{ex:doyle} Let $0<l<n$ be integers such that $\ggd(l,n)=1$.
  Let $V$ be the standard representation of $G=\Sl(l)$ and put $W=V^n$.

  Then $R:=\Sym(W)^G$ is the homogeneous coordinate ring of the Grassmannian $\Gr(l,n)$
  for the Pl\"ucker embedding.
 Let $P_{l,n-l}$ be the set of partitions whose young tableaux are  above the diagonal in a rectangle of size $l\times (n-l)$.
 In loc.\ cit.\  it is shown that the following module of covariants defines an NCCR for $R$:
 \[
   M=\bigoplus_{\lambda\in P_{l,n-l}} M(S^\lambda V)
 \]
(compare with \eqref{eq:det}).
\end{example}
We reiterate that even if an NCCR exists, there doesn't have to be one given by a module of covariants. See Example \ref{ex:toric}.
\subsection{NCCRs via crepant resolutions obtained by GIT}
\label{ssec:gitcrepant}
Here we discuss some results from \cite{HLSam} that shows that in certain cases the  NCCRs for the categorical quotients $X\quot G$ that we constructed
in \S\ref{sec:quotient} can be obtained as the endomorphisms of a tilting bundle on a crepant resolution,  i.e. the method of \S\ref{ssec:tilting}.
This crepant resolution is constructed using a geometric invariant theory. It turns out that we have to allow crepant resolutions
by Deligne-Mumford stacks. This occurred already before in Lemma \ref{prop:dm} and Example \ref{ex:toric}.
\begin{remark}
To construct a resolution of $X\quot G$ using geometric invariant theory, one needs a linearized line bundle on $X$. Since here $X$ is a representation,
the only $G$-equivariant line bundles on $X$ are those obtained from characters of $G$. If $G$ is semi-simple
then there are no (non-trivial) characters so we cannot proceed. Thus we can for example 
not deal with determinantal varieties of symmetric and skew-symmetric matrices (see Theorems \ref{ref-1.4.3-14}  and \ref{ref-1.4.5-15}).
In those cases the relevant groups were respectively $\Sp(2l)$ and the connected component $\operatorname{SO}(l)$ of $\operatorname{O}(l)$, both of which are semi-simple. On the other hand, ordinary determinantal
varieties are fine since in that case $G=\GL(l)$ which has a non-trivial character given by the determinant, which may be used to construct a crepant resolution.
\end{remark}
\begin{remark}
Geometric invariant theory is still helpful for constructing a resolution of $X\quot G$ via a procedure invented by Kirwan \cite{Kirwan1,Reichstein,EdidinRydh}.
However these resolutions 
are usually not crepant. Consistent with expected minimality of NCCRs (see \S\ref{ssec:partial}) we are able to show that some NCCRs embed inside them \cite{SVdBKirwan}.
\end{remark}
We retain the notations and assumptions of the previous section. We assume that $W$ is a quasi-symmetric representation
of $G$ and $X=\Spec \Sym(W)=W^\ast$.
Recall that for a character $\mu\in X(G)=X(T)^\Wscr$ we may define a $G$-invariant open subset of $X$ as
\[
  X^{ss,\mu}=\{x\in X\mid \exists k>0\text{ and }s\in \Gamma(\Oscr_X\otimes \mu^{k})^G\text{ such that }s(x)\neq 0\}.
  \]
The variety  $X^{ss,\mu}$ admits a good quotient\footnote{A $G$-equivariant map $Z\rightarrow Y$ is a good quotient if locally on $Y$ it is of the form $U\rightarrow U\quot G$ for
  $U$ affine.} $X^{ss,\mu}\quot G$ which is proper over $X\quot G$. For $U$ a representation of $G$ we write $\Mscr(U)$ for the
vector bundle on $X^{ss,\mu}/G$ given by $U\otimes \Oscr_{X^{ss,\mu}/G}$.
The global sections of $\Mscr(U)$ are equal to $M(U)$.

  Below we let $\Hscr_0$ be the \emph{central} hyperplane arrangement on $X(G)_\RR=X(T)^\Wscr_\RR$ corresponding to the affine hyperplane arrangement $\Hscr$
  introduced in \eqref{eq:hyper}. Thus the hyperplanes in $\Hscr_0$ are the hyperplanes which are induced from central hyperplanes in $X(T)_\RR$ which are parallel to the facets
  of $\bar{\Sigma}$. 

  \begin{proposition}[\protect{\cite[Proposition 2.1]{HLSam}}]
    Assume that the action of $T$ on $X$ has generically finite stabilizers and let $\mu\in X(G)$ be in the complement of $\Hscr_0$. Then
    $X^{ss,\mu}/G$ is a Deligne-Mumford stack.
  \end{proposition}
  \begin{lemma} Assume that $(G,W)$ is generic. Then the canonical map $X^{ss,\mu}/G\rightarrow X\quot G$ is crepant.
  \end{lemma}
  \begin{proof} This is proved in \cite[Lemma 4.5]{SVdB4} in the case that $G$ is a torus, but this assumption is not relevant for the proof. 
  \end{proof}
The following is one of the main results of \cite{HLSam}. It is proved using similar combinatorics as in \cite{SVdB}.
  \begin{theorem}[\protect{\cite{HLSam}}]
    \label{th:HLS1}
    Assume that the action of $T$ on $X$ has generically finite stabilizers and let $\mu\in X(G)$ be in the complement of $\Hscr_0$. Let
    $U_\delta$ be as in the statement of Theorem \ref{th:mainth}. Then $\Mscr(U_\delta)$ is a tilting bundle on $X^{ss,\mu}/G$ such that
    $\End_{X^{ss,\mu}/G}(\Mscr(U_\delta))=M(\End(U_\delta))$.
  \end{theorem}
  \begin{proof} This follows from combining \cite[Theorem 1.2]{HLSam} with \cite[Lemma 2.9]{HLSam}.
  \end{proof}
  In this way we obtain more evidence for Conjecture \ref{con:BO}.
  \begin{corollary}[\protect{\cite[Corollary 1.3]{HLSam}}] \label{cor:GIT}
    Under the hypotheses of Theorem \ref{th:HLS1} if $\mu,\mu'\in X(G)$ are in the complement of $\Hscr_0$
    and the complement of $\Hscr$ is non-empty (i.e.\ $\Hscr$ is non-degenerate)
    then $\Dscr(X^{ss,\mu}/G)\cong \Dscr(X^{ss,\mu'}/G)$.
  \end{corollary}
  We also obtained the promised description of NCCRs via resolutions.
  \begin{corollary} Assume that $(G,W)$ is generic and let $\Lambda$ be an NCCR constructed via Theorem \ref{th:mainth}. Let $\mu$ be in the complement of $\Hscr_0$. Then $\Lambda$ is the endomorphism ring
    of a tilting bundle on the DM stack $X^{ss,\mu}/G$.
  \end{corollary}
  \begin{remark}
    Hidden behind what is discussed in \S\ref{ssec:covariants},\S\ref{ssec:gitcrepant} is the idea of  \emph{windows}, pioneered in \cite{SegalDonovan}. This is  based on the
    fact that we have a restriction map
    \[
\Res:      \Dscr(X/G)\rightarrow \Dscr(X^{ss,\mu}/G),
    \]
    It is then natural to try to find a full subcategory $\Dscr\subset  \Dscr(X/G)$ such that the restriction of $\Res$ to $\Dscr$ yields an equivalence
    $\Dscr \cong \Dscr(X^{ss,\mu}/G)$. A very general result in this direction is \cite[Theorem 1.1]{HL}, see also \cite{BFK}.

    In concrete cases one may hope to define $\Dscr$ as the full subcategory of $\Dscr(X/G)$ which is split generated by $U\otimes \Oscr_{X/G}$ for a suitable
    $G$-representation $U$ whose
    highest weights are restricted to a certain subset $\Lscr$ of  $X(T)^+$ (a ``window''). This is precisely what happens in Theorem \ref{th:HLS1},
    where we take $\Lscr=\Lscr_\delta$. The resulting category $\Dscr$ is a concrete realization of \cite[Theorem 1.1]{HL}, see \cite[Lemma 3.5]{HLSam}.

    One does not actually need to have non-trivial $X^{ss,\mu}/G$ to apply the window principle. The proof of Theorem \ref{th:mainth},
    is based on the fact that $\mod(\Lambda_\delta)$ embeds in $\coh(X/G)$ as the abelian category with a projective generator $U_\delta\otimes \Oscr_{X/G}$.
\end{remark}
\subsection{Local systems, the SKMS and schobers}
\label{sec:SKMS}
In this slightly informal section we assume that the hypotheses of Theorem \ref{th:HLS1} hold.
While Corollary \ref{cor:GIT} implies that two different $\Dscr(X^{ss,\mu}/G)$, $\Dscr(X^{ss,\mu'}/G)$ are derived equivalent, the actual derived equivalence
depends on the choice of $\delta$ in the complement of~$\Hscr$. Moreover by considering compositions
$\Dscr(X^{ss,\mu}/G)\xrightarrow{\delta} \Dscr(X^{ss,\mu''}/G)\xrightarrow{\delta'} \Dscr(X^{ss,\mu'}/G)$ we may produce more derived equivalences. This is consistent with the assertion in Remark \ref{rem:notexplicit}
that there is no ``god-given'' derived equivalence between different crepant resolutions. A different way of saying this is that a crepant resolution may have
a large group of auto-equivalences.

 If $M$ is a Calabi-Yau variety then homological mirror symmetry predicts the existence of a space $S$ (the "stringy K\"ahler moduli space'' or SKMS) such that $\pi_1(S)$ acts on $\Dscr(M)$. More precisely $S$ is the moduli space of complex structures on  the mirror dual $M^\vee$ of $M$. In many cases there are
 good heuristic descriptions of $M^\vee$ and $S$. 

 Even without access to the full mirror symmetric context, which may be technically challenging or even only heuristic, it turns out to be very illuminating to represent the derived autoequivalences of an algebraic variety (or stack)
 as elements of $\pi_1(S)$ for a suitable topological space $S$. Alternatively we may think of such a representation as a \emph{local system of triangulated
   categories on $S$}.
 Understanding this for $\Dscr(X^{ss,\mu}/G)$ was, according to the authors, one of the main motivations
 for writing \cite{HLSam}. Indeed when $X$ is a quasi-symmetric representation, under hypotheses of Theorem \ref{th:HLS1} one may take
 \[
   S= (X(G)_{\CC}-\Hscr_\CC)/X(G)
 \]
 where $\Hscr_\CC$ denotes the complexification of the real hyperplane arrangement $\Hscr$ \cite[Proposition 6.6]{HLSam}.

 In this case there is a nice way to understand that action of $\pi_1(S)$ on $\Dscr(X^{ss}/G)$ \cite[\S6]{HLSam}\cite{SVdB10}. Using Theorem \ref{th:HLS1} again we may just as
 well describe the action of $\pi_1(S)$ on $\Dscr(\Lambda_\delta)$ for $\delta$ contained in the complement of $\Hscr$ and $\Lambda_\delta=M(\End(U_\delta))$.

 For $\delta\in X(G)_\RR$ define $U_\delta$ as in \eqref{eq:Udelta} and put $\Dscr_\delta=\Dscr(\Lambda_\delta)$. Now $\Hscr$ defines a cell decomposition of $X(G)_\RR$
 and it is easy to see that $U_\delta$ only depends on the cell to which $\delta$ belongs. 
 Hence for a cell $C$ let us write $\Lambda_C:=\Lambda_\delta$, $\Dscr_C:=\Dscr_\delta$ for $\delta\in C$. We will refer to the cells of maximal
 dimension as chambers. These are also the connected components of the complement of $\Hscr$.

 If $C'$ is a face of $C$ then there is an idempotent $e_{C,C'} \in \Lambda_{C'}$ such that $\Lambda_{C}=e_{C,C'}\Lambda_{C'} e_{C,C'}$.
 If $C\neq C''$ are distinct adjacent chambers, sharing
 a codimension one face $C'$ then the functor $e_{C'',C'}\Lambda_{C'}  e_{C,C'} \Lotimes_{\Lambda_C} -$ defines an equivalence of categories $\phi_{C,C''}:\Dscr_C\rightarrow \Dscr_{C''}$.

 Put $\tilde{S}=X(G)_\CC-\Hscr_\CC$ and let $\Pi_1(\tilde{S})$ be the groupoid whose objects are the chambers and whose morphisms are given by the homotopy classes of paths in $X(G)_\CC-\Hscr_\CC$ connecting
 the chambers. Then $\Pi_1(\tilde{S})$ is equivalent to the fundamental groupoid of $X(G)_\CC-\Hscr_\CC$.
 If $C$, $C''$ are adjacent chambers separated by a hyperplane $H\in \Hscr$, such that $H(C'')>0$ then there is a canonical (up to homotopy) minimal path $\ggamma_{C,C''}$  in $X(G)_\CC-\Hscr_\CC$ going from $C$ to $C''$ and passing through $\{\Imm H_\CC>0\}$. Sending~$C$ to $\Dscr_C$ and $\ggamma_{C,C''}$ to  $\phi_{C,C''}$ defines a representation of the groupoid $\Pi_1(\tilde{S})$ in triangulated categories.

 If $\chi\in X(G)$ then tensoring by $\chi$ defines an equivalence $\Dscr_{C}\r \Dscr_{C+\chi}$ and in this way the representation of $\Pi_1(\tilde{S})$ may be
 extended to a representation of $\Pi_1(\tilde{S})\rtimes X(G)$ and the latter is equivalent to the fundamental groupoid $\Pi_1(S)$ of $S=\tilde{S}/X(G)$
 \cite[Chapter 11]{BrownRonald}\cite[\S6]{HLSam}. Hence fixing a ``base chamber'' $C$ we get an action of $\pi_{1,C}(S)$ on $\Dscr_C$.

 \begin{remark}
   It is shown in \cite{SVdB10} that
   the family of triangulated categories $(\Dscr_C)_C$ for \emph{all} cells $C$ is a so-called $X(G)$-equivariant \emph{perverse schober}.
  This is a categorification of
 a perverse sheaf on $X(G)_\CC/X(G)$   \cite{KapranovSchechtmanSchobers}  (see also \cite{Donovan, BondalKapranovSchechtman}). Note that $X(G)_\CC/X(G)$ is a torus and $S$ is the complement of a ``toric hyperplane arrangement''.
 If $G$ is itself a torus $T$ then $X(T)_\CC/X(T)$ may be
 identified with the dual torus $T^\vee$.
 \end{remark}
 \begin{remark}
   The $X(G)$-equivariant hyperplane arrangement constructed by Halpern-Leistner and Sam in \cite{HLSam}
   is very similar to the $\operatorname{Cl}(R)$-equivariant hyperplane arrangement associated to a 3-dimensional terminal
   complete Gorenstein ring $R$  constructed by Iyama and Wemyss (see Example 4.12).
   One would expect there to be an associated equivariant schober also in this case.  In the case of a single curve flop this is essentially contained in
   \cite[\S3]{DWmutations}.
 \end{remark}
 \begin{remark} As explained we have an action of $\pi_1(S)$ on $\Dscr_C$ for a chamber $C$ and hence also an action of $\pi_1(S)$ on $K_0(\Dscr_C)_\CC$.
   In other words we have a local system $L$ on $S$. It is then a natural question if this local system occurs as the solutions of a natural system of differential
   equations. In the case that $G$ is a torus we show in \cite{SVdB1000} that a generic ``equivariant'' deformation of $L$ is obtained as the solution
   of a well-known system of differential equations introduced by Gel'fand, Kapranov and Zelevinsky \cite{GKZhyper}.
   This starts from a computation by Kite \cite{Kite} which shows that the hyperplane arrangement constructed in \cite{HLSam} is up to translation defined
   by the so-called ``principal $A$-determinant'', an important ingredient in the theory developed   Gel'fand, Kapranov and Zelevinsky.
   For more information see \cite{SpelaECM}.
 \end{remark}
 \begin{remark}
The themes touched upon in this section occur in many different contexts. See e.g.\ \cite{BHMB,ABM,BridgelandYuSutherland}.
\end{remark}
\section{NCCRs and stringy $E$-functions}
\label{sec:stringy}
In this section we discuss some ongoing work of Timothy De Deyn (see \cite{TDD}). Let~$X$ be an algebraic variety over $\CC$. The cohomology groups $H^i_c(X,\CC)$ carry a natural mixed Hodge structure. We denote
by $h^{p,q}(H^i_c(X,\CC))$ the dimension of the $(p,q)$-type component of $H^i_c(X,\CC)$. The \emph{Hodge polynomial} of $X$ is defined by
\[
E(X,u,v)=  \sum_{p,q,i} (-1)^i h^{p,q}(H^i_c(X,\CC))u^p v^p.
\]
The Hodge polynomial defines a ring homomorphism from the Grothendieck ring of algebraic varieties $K_0(\operatorname{Var}/\CC)$ to $\ZZ[u,v]$.

We put $e(X)=E(X,1,1)$; i.e.
\[
  e(X)=\sum_{i} (-1)^i \sum_{p,q} h^{p,q}(H^i_c(X,\CC))=\sum_i (-1)^i\dim H^i_c(X,\CC).
\]
In other words $e(X)$ is the \emph{Euler characteristic} (with compact support\footnote{If $X$ is smooth then, by Poincar\'e duality, the Euler characteristic with
  compact support coincides with the usual Euler characteristic $\sum_i (-1)^i\dim H^i(X,\CC)$.}) of $X$. It defines a ring homomorphism from $K_0(\operatorname{Var}/\CC)$ to $\ZZ$.
\begin{definition}[\protect{\cite[Definition 3.1]{Batyrev1}}]
  Assume that $X$ is a normal $\QQ$-Gorenstein algebraic variety$/\CC$ with at most log-terminal singularities and let $\pi:Y\rightarrow X$ be a resolution of singularities
whose exceptional locus is a normal crossing divisor. Let $D_1,\ldots,D_r$ be the irreducible components of the exceptional locus and put $I=\{1,\ldots,r\}$.
For any subset~$J\subset I$ we set   $D_J=\bigcap_{j\in J} D_j$, $D^\circ_J:=D_J\setminus \bigcup_{j\in I\setminus J} D_j$. The \emph{stringy $E$-function} of
$X$ is defined as
\begin{equation}
  \label{eq:stringy}
  E_{st}(X,u,v):=\sum_{J\subset I} E(D^\circ_J,u,v) \prod_{j\in J} \frac{uv-1}{(uv)^{a_j+1}-1}
\end{equation}
where the numbers $a_j\in \QQ\,\cap\, ]-1,\infty[$ are defined by
\[
  K_Y=\pi^\ast K_X+\sum_{j=1}^r a_j D_j.
  \]
\end{definition}
Putting $e_{st}(X)=\lim_{u,v\rightarrow 1} E_{st}(X,u,v)$ defines the \emph{stringy Euler characteristic} of $X$, with the formula
\[
  e_{st}(X)=\sum_{J\subset I} e(D^\circ_J) \prod_{j\in J} \frac{1}{a_j+1}.
\]

It follows from the theory of motivic integration (see e.g.\ \cite{Kon95,  DL1,Batyrev, CrawMotivic, Le_n_Cardenal_2020}) that $E_{st}(X,u,v)$ is independent of the
chosen resolution $Y$ \cite[Theorem 3.4]{Batyrev1}. Indeed $E_{st}(X,u,v)$ may be obtained by integrating over the arc space associated to $X$ \cite{DL1}. In a similar vein $E_{st}(X,u,v)=E_{st}(Y,u,v)$ holds for birational maps $\pi:Y\rightarrow X$ satisfying $\pi^\ast K_X=K_Y$ \cite[Theorem 3.12]{Batyrev1}.

\medskip

If $X$ is smooth then
the stringy $E$-function coincides with the Hodge polynomial. Hence one has: 
\begin{theorem}[\protect{\cite[Theorem 3.12]{Batyrev1}}] \label{th:stringy} 
If $X$ has a crepant resolution $Y$ then
the stringy $E$-function of~$X$ coincides with the Hodge polynomial of $Y$. In particular, it is a polynomial. Similarly, the stringy Euler characteristic of
$X$ coincides with the usual Euler characteristic of~$Y$ and hence it is an integer.
\end{theorem}
The following conjecture seems natural:
\begin{conjecture}[\protect{\cite{TDD}}]
  \label{qu:E}
  If $X$ is a normal Gorenstein variety$/\CC$ with an NCCR then its stringy $E$-function is a polynomial.
\end{conjecture}
We give some evidence for this conjecture below, but at this point it is probably safer to
regard it as a question. We illustrate below in Remark \ref{rem:weakly} and Example \ref{ex:twisted} that reasonable extensions of this conjecture
are false.
\begin{example}
  \label{ex:finite}
  Quotient varieties of the form $\CC^n\quot G$ for $G\subset \Sl(n)$ finite always have a stringy $E$-function which is a polynomial by \cite{BaDa}\cite[Theorem 3.6]{DL2}.
  They also have an NCCR by \S\ref{ssec:qs}. So in this case Conjecture \ref{qu:E} is true.
\end{example}
\begin{example}
  \label{ex:battoric}
Batyrev proves in \cite[Proposition 4.4]{Batyrev} that the stringy $E$-function of any toric variety with Gorenstein singularities is a polynomial. Hence
Conjecture \ref{qu:E} is compatible
with  Conjecture \ref{conj:nccrtoric}.
\end{example}
  A good test for Conjecture \ref{qu:E} is given by cones over Fano varieties.
  \begin{proposition}[\protect{\cite{TDD}}]
    \label{prop:cone} Let $Z$ be a smooth projective variety$/\CC$ with ample line bundle $\Oscr_Z(1)$ and let $X=\Spec R$ be the corresponding cone.
  Assume $\omega_Z=\Oscr_Z(-n)$ for $n\ge 1$.  Then $R$ is Gorenstein
and
\begin{equation}
  \label{eq:Est}
   E_{st}(X,u,v)=E(Z,u,v) \frac{(q-1)q^n}{q^n-1}
 \end{equation}
 with $q=uv$.  In particular:
 \begin{equation}
   \label{eq:estformula}
   e_{st}(X)=\frac{e(Z)}{n}.
 \end{equation}
\end{proposition}
\begin{example} \label{ex:timothy}
  Consider the Grassmannian
  $Z:=\Gr(d,n)$. 
  Then (e.g.\ \cite[Proposition A.4]{BoLi})
  \begin{equation}
    \label{eq:Estgr}
    E_{st}(Z,u,v)= {n\choose d}_q
  \end{equation}
  where
  \[
    {n\choose d}_q=\frac{(q^n-1)(q^{n-1}-1)\cdots (q^{n-d+1}-1)}{(q-1)(q^2-1)\cdots (q^d-1)}.
  \]
  Hence
  \begin{equation}
    \label{eq:chooseformula}
    e_{st}(Z)={n\choose d}.
  \end{equation}
  Let $X$ be the cone over $Z$ with respect to the Plucker embedding and let $R$ be the coordinate ring of $X$.
  Using \eqref{eq:Est}\eqref{eq:Estgr} it is shown in \cite{TDD} that in this case $E_{st}(X,u,v)$   is a polynomial precisely when $\ggd(d,n)=1$. 
  On the other hand   by \cite[Theorem 3.11]{Doyle} (see Example \ref{ex:doyle} above) $R$ has an NCCR when $\gcd(d,n)=1$. So  Conjecture \ref{qu:E} is true in this case.
  \end{example}
\begin{remark}
  \label{rem:weakly}
  It is shown in in \cite[\S4.2]{TDD} that the cone over an arbitrary Grassmannian $\Gr(d,n)$ always has a weakly crepant categorical resolution (see \S\ref{sec:ccr}).  Hence it follows
   that Conjecture \ref{qu:E} is false for weakly crepant categorical resolutions. On the other hand it seems reasonable to extend
   Conjecture \ref{qu:E} to strongly crepant categorical resolutions.
\end{remark}
\begin{remark} In view of Theorem \ref{th:stringy} one may naively ask if it is true that $E_{st}(X,u,v)$  being a polynomial implies that $X$ has a crepant resolution. Not unexpectedly this fails drastically. Finite group quotients
  and affine toric varieties have a polynomial stringy $E$-function, as we have seen above, but they don't need to have a crepant resolution.
  In fact, the example $\CC^4\quot \ZZ_2$ given in  \S\ref{ssec:qs} of a Gorenstein singularity with an NCCR but
  without a crepant resolution, lives in both classes. On the other hand these classes of counterexamples are not very convincing since they admit crepant resolutions by smooth Deligne Mumford 
  stacks, which is just as good. In the case of finite quotient singularities this is clear, and for toric varieties it follows from Lemma~\ref{prop:dm}.

  In contrast one may show that for $X$ as in Example \ref{ex:timothy} there is no crepant resolution by a smooth DM-stack (this is mainly because $R$ is factorial).
  So in some sense it is a ``better'' counterexample (when $\gcd(d,n)\neq 1$).
  \end{remark}
GIT quotients form an important
class of toric varieties (see \cite[Corollary 14.2.16]{CoxLittleSchenck}).
So in view of Example \ref{ex:battoric} as well as Example \ref{ex:finite}
the following question by Batyrev suggests itself.
  \begin{question}[\protect{\cite[Question 5.5]{Batyrev1}}]
    Does a GIT quotient of $\CC^n$ for a linear action of
    $G \subset \Sl(n)$ always have a stringy $E$-function that is a polynomial?
\end{question}
Alas the answer
is negative. Indeed Example \ref{ex:timothy} for $\ggd(d,n)\neq 1$ gives a simple counterexample since the cone over a Grassmannian
is a GIT quotient for $\Sl(d)$ acting on $n$ copies of its standard representation.

\medskip

The  first counterexample however was constructed much earlier in \cite{MR1953528}.
\begin{example}
  \label{ex:twisted}
  Let $W$ be given by three copies of the adjoint representation of $G=\Sl(2)$. Then by a quite involved computation it is shown in \cite[Corollary 1.2]{MR1953528} that the stringy $E$-function
  of $\Spec R$ for $R=\Sym(W)^G$ is not a polynomial.
  This example is interesting since by \cite[Theorem 1.4.5]{SVdB} $R$ has a twisted NCCR.
  In other words, Conjecture \ref{qu:E} is also false for twisted NCCRs.

  As a side remark we  note that this twisted NCCR is a rather classical object. It is the trace ring generated by 3 generic traceless $2\times 2$ matrices  \cite{Ar4, Raz,  Proc,Proc2,Leb3, LeBrVdB}.  
\end{example}
In the setting of Theorem \ref{th:stringy} the Euler characteristic of $Y$ can be computed using periodic cyclic homology, thanks to the Hochschild-Kostant-Rosenberg theorem. Below we define the Euler characteristic $e(\Lambda)$ of an algebra $\Lambda$ or a sheaf of algebras as
$\dim\HP^{\text{even}}(\Lambda)-\dim\HP^{\text{odd}}(\Lambda)$. If $\Lambda$ is a quasi-coherent
sheaf of algebras then we use $\HP^{*}(\Lambda):=\HP^{*}(\Perf_{dg}{\Lambda})$ (where $\Perf_{dg}(\Lambda)$ is a standard dg-enhancement of $\Perf(\Lambda)$). 
The following conjecture appears plausible.
\begin{conjecture}
  \label{conj:periodic}
 The stringy Euler characteristic of a normal Gorenstein variety can be computed as the Euler characteristic of an NCCR, computed via periodic cyclic homology.
\end{conjecture}

\begin{remark} One can again not expect this conjecture to hold for twisted NCCRs. An interesting example is given in \cite{BorisovWang}. It
  was shown by \cite{BonOrl,MR2419925} that for a generic complete intersection $Y$ of $n$ quadrics in $\PP^{2n-1}$ one has a derived equivalence
  between $Y$ and $(\PP^{n-1},\Bscr_0)$ where $\Bscr_0$ is the even part of the universal Clifford algebra corresponding to the quadrics defining $Y$.
  Because of the derived equivalence we then have $e(Y)=e(\Bscr_0)$ \cite{kellerexact}. One may show that $\Bscr_0$ is a twisted NCCR of its center which is a double cover $Z$ of $\PP^{n-1}$ \cite[\S1]{BorisovWang}. 
  It is shown in \cite{BorisovWang} that in general $e(Y)\neq e_{st}(Z)$ and hence $e(\Bscr_0)\neq e_{st}(Z)$. So Conjecture \ref{conj:periodic} does not extend to
  twisted NCCRs.
\end{remark}

In suitable ``local'' contexts
(e.g.\ \cite[Theorem 9.1]{MR3338683})
Conjecture \ref{conj:periodic} leads to a more  concrete conjecture:\footnote{To handle the complete case one has to use a ``completed'' version of periodic cyclic homology. See \cite{MR3338683}}
\begin{conjecture}[\protect{\cite{TDD}}]
  \label{con:counting}
  Let $R$ a normal Gorenstein ring  which is either a complete local ring, or else connected $\NN$-graded (i.e.\ $R_0=\CC$).
  Assume that $R$ has an NCCR $\End_R(M)$. Then the number of non-isomorphic indecomposable  summands
  of $M$ is equal to $e_{st}(X)$ for $X=\Spec R$.
\end{conjecture}
\begin{example} If $X=\Spec R$ is an affine toric variety as in \S\ref{sec:toric}, with Gorenstein singularities, then Batyrev \cite[Proposition 4.10]{Batyrev1} proves that
  $e_{st}(X)$ is equal to the volume of the associated polytope $P$ (see \S\ref{sec:toric}). So Conjecture \ref{con:counting} is compatible with Conjecture \ref{conj:volume}. 
\end{example}
\begin{example}
  \label{ex:finite2}
  For varieties of the form $X=W\quot G=\Spec \CC[W]^G$ for $G\subset \Sl(n)$ finite and $W$ a finite dimensional representation of $G$ it follows from
  \cite[Theorem 8.4]{Batyrev2} that $e_{st}(X)$ is equal to the number of conjugacy classes
  in $G$. This number is in turn equal to the number of irreducible representations of $G$ and hence equal to the number of non-isomorphic indecomposable
  summands of the reflective $\CC[W]^G$-module $\CC[W]$ which defines an NCCR for $\CC[W]^G$ by \S\ref{ssec:qs}. So Conjecture \ref{con:counting} is true
  in this specific example.
\end{example}

\begin{example} \label{ex:catalan}
  Let $X$ be the cone over $\Gr(d,n)$ as in Example \ref{ex:timothy}. Then by \eqref{eq:estformula}\eqref{eq:chooseformula} we have
 \[
   e_{st}(X)=\frac{1}{n}{n\choose d}.
 \]
We check that Conjecture \ref{con:counting} is compatible with the NCCR constructed by Doyle in \cite[Theorem 3.11]{Doyle} (see Example \ref{ex:doyle} above). 
Conjecture \ref{con:counting} amounts to
 \[
   |P_{d,n-d}|=\frac{1}{n}{n\choose d}.
 \]
 which is indeed true by \cite[\S12.1]{Loehr}.
\end{example}
Example \ref{ex:catalan} can be put in a more general context. Let us first state a lemma.
\begin{lemma} \label{lem:rkK0}
  Let $Z$ be a smooth projective variety with a tilting complex. Then one has $e(Z)=\rk K_0(Z)$.
\end{lemma}
\begin{proof}
  Let $\Tscr$ be the tilting complex and put $A=\End_Z(\Tscr)$.
We have:
\begin{enumerate}
\item Euler characteristics may be computed with periodic cyclic homology.
\item Periodic cyclic homology is invariant under derived equivalence \cite{kellerexact}, and so are Grothendieck groups;
\item $\HP^\ast(A)=\HP^\ast(A/\rad A)$ by Goodwillie's theorem \cite[Theorem II.5.1]{Goodwillie}, and the standard fact that $K_0(A)=K_0(A/\rad A)$.
\end{enumerate}
So we conclude
\[
  e(Z)\overset{(1,2)}{=}e(A)\overset{(3)}{=}\rk K_0(A)\overset{(2)}{=}\rk K_0(Z).\qedhere
\]
\end{proof}
Let us go back to the setting of Proposition \ref{prop:cone} but assume now in addition that $Z$ has a tilting complex.
Then by \eqref{eq:estformula} combined with Lemma \ref{lem:rkK0}
we get
  \[
    e_{st}(X)=\frac{\rk K_0(Z)}{n}.
  \]
  Assuming that $R$ has a graded NCCR $\Lambda$ then Conjecture \ref{con:counting} implies
  \begin{equation}
    \label{eq:frac}
    \rk K_0(\Lambda)=\frac{\rk K_0(Z)}{n}.
  \end{equation}
  \begin{example}
    \label{ex:motivic} This formula holds for the NCCRs constructed via Proposition~\ref{prop:Kuz}. Indeed $\rk K_0(\Lambda)$ is given by the number $u$ of non-isomorphic indecomposable
    summands of $\Escr$. On the other hand $\Dscr(Z)$ has a semi-orthogonal decomposition consisting of $n$ parts whose $K_0$ also has rank $u$. So \eqref{eq:frac} does indeed hold
    and we obtain again some evidence for Conjecture \ref{con:counting}. 
  \end{example}
  \begin{remark}
    One way to think of this example as the realization of the (conjectured) ``motivic'' identity \eqref{eq:frac}  via semi-orthogonal decompositions
    of derived categories. See \cite{PolVdB} for another (deeper) instance of this principle. 
  \end{remark}

  \begin{ack}
    First and foremost I am grateful to my coauthor and friend \v{S}pela \v{S}penko for contributing much of the mathematics of our joint work. Without her input this survey would have
    been a lot shorter.
 
Furthermore    I thank Shinnosuke Okawa for readily answering all my questions about the minimal model program. Likewise I thank Michael Wemyss for input on the non-commutative geometry of cDV singularities.
\end{ack}

\begin{funding}
This project has received funding from the European Research Council (ERC) under the European Union's Horizon 2020 research and innovation programme (grant agreement No 885203).
\end{funding}
\def\Spenko{\v{S}penko } \def\Sp{Spe}
\providecommand{\bysame}{\leavevmode\hbox to3em{\hrulefill}\thinspace}
\providecommand{\MR}{\relax\ifhmode\unskip\space\fi MR }
\providecommand{\MRhref}[2]{%
  \href{http://www.ams.org/mathscinet-getitem?mr=#1}{#2}
}
\providecommand{\href}[2]{#2}

\end{document}